\documentclass[final,leqno]{siamltex}
\usepackage{booktabs} 
\usepackage[noadjust]{cite}
\usepackage{amsmath}
\allowdisplaybreaks[4]
\usepackage{multirow}
\usepackage{graphicx}
\usepackage{amssymb}
\usepackage{mathtools}
\usepackage{mathrsfs}
 
\usepackage{latexsym, bm}
\usepackage{hyperref}
\numberwithin{equation}{section}
\newtheorem{remark}{Remark}[section]

\newcommand{\vertiii}[1]{{\left\vert\kern-0.23ex\left\vert\kern-0.23ex\left\vert #1
		\right\vert\kern-0.23ex\right\vert\kern-0.23ex\right\vert}}
\renewcommand\abstractname{Abstract}

\renewcommand\figurename{Figure}
\renewcommand\tablename{Table}

\normalsize

\renewcommand\refname{References}
\def\no{\nonumber}
\def\ba{\begin{equation}\begin{aligned}}
\def\ed{\end{aligned}\end{equation}}

\begin{document}
\renewcommand{\thefootnote}{\fnsymbol{footnote}}
\title{A Second-Order Maximum-Bound-Preserving and Modified-Energy-Stable Exponential Time-Differencing Method for Allen--Cahn-Type Gradient Flows}

\author{Wenshuai Hu\thanks{Laboratory of Mathematics and Complex Systems, Ministry of Education and School of Mathematical Sciences, Beijing Normal University, Beijing 100875, China.
   }
    \and
    Guanghua Ji\footnotemark[1]
    \thanks{Corresponding author. Email: \texttt{ghji@bnu.edu.cn}}
    \and Xiao Li\footnotemark[1]
}


\maketitle
\begin{abstract}
The energy dissipation law and the maximum bound principle (MBP) are two fundamental structural properties of the Allen--Cahn equation.
In this paper, we develop and analyze a second-order linear numerical scheme
for a class of scalar Allen--Cahn-type gradient flows. The scheme combines the
generalized scalar auxiliary variable (GSAV) approach with a stabilized
second-order exponential time-differencing discretization. Assume that the
initial data are bounded by some $\beta>0$, that
$f(\beta)\le0\le f(-\beta)$, and that
$\kappa\ge\|f'\|_{L^\infty([-\beta,\beta])}$. We prove that the resulting
GSAV--ETD2 scheme preserves the MBP for arbitrary time step sizes and
dissipates a modified discrete energy.
In addition, rigorous error analysis is carried out for the proposed scheme, establishing second-order accuracy in both time and space without imposing any coupling condition between the time step $\tau$ and the spatial mesh size $h$.
We also present numerical experiments whose results are consistent with the predicted accuracy, the MBP, and the modified energy dissipation law.
\end{abstract}



\pagestyle{myheadings}
\thispagestyle{plain}
\markboth{W. Hu, G. Ji, and X. Li}{GSAV--ETD2 method for the Allen--Cahn equation}
\section{Introduction}
Gradient flows provide a fundamental variational framework for describing
dissipative dynamics in phase-field models, interfacial motion, materials
science, and complex fluids \cite{ShenYang2010,shen2019new}. Among them, Allen--Cahn-type equations form an
important class of gradient-flow models \cite{allen1979microscopic} and can be written in the general form
\begin{equation}
u_t
=
\varepsilon^2\Delta u+f(u),
\qquad
t>0,\quad \boldsymbol{x}\in\Omega,
\label{eq:allen_cahn_model}
\end{equation}
where the spatial domain $\Omega\subset\mathbb{R}^d$ and
$u:[0,\infty)\times\overline{\Omega}\to\mathbb{R}$
is the unknown function, $\varepsilon>0$ is an interfacial parameter,
and
$f(u)$
is the nonlinear reaction. Equipped with the periodic or homogeneous
Neumann boundary condition, \eqref{eq:allen_cahn_model} can be viewed as the
$L^2$ gradient flow with respect to the energy functional
\begin{equation}
E(u)
=
\int_\Omega
\left(
\frac{\varepsilon^2}{2}|\nabla u(\boldsymbol{x})|^2
+
F(u(\boldsymbol{x}))
\right)
\,\mathrm d\boldsymbol{x},
\label{eq_allen_cahn_energy}
\end{equation}
where
$F$ is a potential function satisfying $f=-F'$. We supplement
\eqref{eq:allen_cahn_model} with the initial condition
$u(0,\boldsymbol{x})=u_0(\boldsymbol{x})$. The energy
\eqref{eq_allen_cahn_energy} is nonincreasing along solutions of
\eqref{eq:allen_cahn_model}, i.e.,
$\frac{\mathrm d}{\mathrm dt}E(u(t))\le0.$

More precisely, assume that $f\in C^2(\mathbb R)$ and that there exists
a constant $\beta>0$ such that
\begin{equation}\label{eq_intro_mbp_condition}
    f(\beta)\le0\le f(-\beta).
\end{equation}
The solution $u$ typically represents the difference between the
concentrations of two phases and is therefore naturally expected to remain
within the interval $[-\beta,\beta]$. This gives rise to the maximum bound principle (MBP): if
$|u_0|\le\beta$ pointwise, then $|u(\boldsymbol{x},t)|\le\beta$ for all $t>0$.
Various extensions of the Allen--Cahn equation, including nonlocal models
for long-range interactions \cite{bates2006some,du2019,du2021} and fractional models for
anomalous diffusion \cite{gui2015traveling}, also possess analogous energy
energy-dissipation laws and MBPs. Since analytical solutions are
generally unavailable, efficient numerical methods that accurately preserve
these intrinsic properties, particularly the energy dissipation law and the
MBP, are highly desirable.


Constructing numerical schemes that simultaneously preserve energy dissipation
and the MBP remains challenging.  A large
body of work has been devoted to energy-stable approximations of gradient
flows, including convex-splitting methods
\cite{eyre1998unconditionally,baskaran2013convergence}, stabilized
implicit--explicit methods \cite{xu2006stability,shen2010numerical,feng2013stabilized},
discrete-gradient methods \cite{du1991numerical,furihata2001stable}, invariant
energy quadratization methods \cite{yang2017linearly,xu2019efficient,yang2020convergence},
and scalar auxiliary variable (SAV) methods
\cite{shen2018scalar,shen2018convergence,shen2019new}.  These approaches are
effective for energy dissipation, but energy stability alone does not imply a
discrete MBP.

Stabilized exponential time differencing (ETD) and integrating-factor methods
offer a natural route to bound preservation because they treat the stiff
linear operator through its semigroup and retain a comparison structure for
the nonlinear part.  First- and higher-order MBP-preserving schemes have been
developed for scalar, nonlocal, and matrix-valued Allen--Cahn equations
\cite{du2019,Li2021,Li2023,liu2025maximum}.  On the other hand, the SAV and
generalized SAV (GSAV) approaches provide an efficient way to construct
linear schemes with unconditional modified energy stability.  Ju et al.
\cite{ju2022generalized} combined GSAV with exponential integration for
Allen--Cahn-type gradient flows.  However, the available analysis of the
original GSAV--EI formulation relies on restrictive control of the time step
and the auxiliary factor, which complicates a fully discrete error analysis.
Moreover, the simultaneous preservation of the modified energy law and the
MBP at second order requires additional structural care.

To address these issues, we develop a stabilized second-order GSAV exponential
time-differencing (GSAV--ETD2) method for gradient flows
satisfying \eqref{eq_intro_mbp_condition}.  The key step is an operator
reformulation of the ETD predictor and corrector.  This representation retains
the semigroup structure needed for the MBP while providing a backward
Euler-type identity that is suitable for discrete energy estimates.  For the scalar
Allen--Cahn specialization, we prove unconditional dissipation of the modified
discrete energy.  If
$\kappa\ge\|f'\|_{L^\infty([-\beta,\beta])}$, both the predictor and the corrected
solution preserve the MBP for arbitrary time step sizes.  Combining the
energy law and the MBP gives uniform two-sided bounds for the GSAV factor,
discrete time-regularity estimates, and a fully discrete error bound of order
$O(\tau^2+h^2)$ under the stated regularity assumptions, with $\tau$ chosen
sufficiently small only to close the discrete Gr\"onwall argument and with no
coupling condition between $\tau$ and $h$.

The main contributions of this work are summarized as follows:
\begin{itemize}
    \item We construct a stagewise linear, stabilized GSAV--ETD2 scheme and
          introduce an operator reformulation that facilitates energy and
          error analysis for its scalar Allen--Cahn specialization.
    \item We establish unconditional modified energy dissipation and an
          unconditional discrete MBP for both the first-order predictor and
          the second-order corrected solution under the explicit structural
          conditions \eqref{eq_intro_mbp_condition} and
          $\kappa\ge\|f'\|_{L^\infty([-\beta,\beta])}$.
    \item For the periodic scalar Allen--Cahn problem with central finite
          differences in space, we derive the optimal estimate
          $\|e_u^n\|_{H_h^1}+|e_s^n|\le C(\tau^2+h^2)$ by combining the energy
          law, the MBP, and discrete time-regularity estimates.
\end{itemize}

The remainder of this paper is organized as follows. Section 2 introduces the
periodic finite-difference discretization, the GSAV reformulation, and the
stabilized GSAV--ETD2 predictor--corrector scheme. Section 3 establishes the
unconditional dissipation law for the modified discrete energy. Section 4
proves the unconditional MBP and derives the uniform
bounds for the GSAV factor
 and discrete time-regularity estimates. Section 5 presents the fully discrete error analysis and proves
second-order convergence in both time and space. Section 6 reports numerical
experiments on temporal and spatial convergence, spinodal decomposition, and
mean-curvature-driven interface motion. Finally, Section 7 concludes the
paper.
\section{Fully discrete GSAV exponential-integrator schemes}\label{section2}
\subsection{Spatial discretization and some preliminaries}
Throughout this paper, we consider the model equation
\eqref{eq:allen_cahn_model} on the
two-dimensional square domain $\Omega=(0,L)\times(0,L)$.
We impose periodic boundary conditions throughout the analysis. Extensions to
other dimensions or boundary conditions require the corresponding discrete
summation-by-parts and semigroup properties and are not pursued here.
In this section, we introduce the notation associated with the spatial
discretization and present several preliminary results for the subsequent analysis.

Given a positive integer $M$, let $h=L/M$ be the uniform mesh size and
define
\[
\Omega_h
=
\left\{
(x_i,y_j)=((i-1)h,(j-1)h)
\,\middle|\,
1\le i,j\le M
\right\}.
\]
For a grid function $v$ on $\Omega_h$, we write
$v_{ij}=v(x_i,y_j)$ and extend it periodically to all integer indices. Let $\mathcal{M}_h$ denote the space of all such periodic grid functions on
$\Omega_h$, namely,
\[
\mathcal{M}_h
:=
\{v:\Omega_h\to\mathbb{R}\},
\qquad
v_{i+kM,j+lM}:=v_{ij},\quad
k,l\in\mathbb Z,\quad 1\le i,j\le M.
\]

For any $v\in\mathcal{M}_h$, the discrete Laplace operator
$\Delta_h$ is defined by
\[
\Delta_h v_{ij}
=
\frac{1}{h^2}
\left(
v_{i+1,j}
+
v_{i-1,j}
+
v_{i,j+1}
+
v_{i,j-1}
-
4v_{ij}
\right),
\qquad
1\le i,j\le M,
\]
and the discrete gradient operator $\nabla_h$ is defined by
\[
\nabla_h v_{ij}
=
\left(
\frac{v_{i+1,j}-v_{ij}}{h},
\frac{v_{i,j+1}-v_{ij}}{h}
\right)^T,
\qquad
1\le i,j\le M.
\]

The eigenvalues of $\Delta_h$ are given by \cite{du2019}
\begin{equation*}
\lambda_{kl}
=
-\frac{4}{h^2}
\left(
\sin^2\frac{k\pi}{M}
+
\sin^2\frac{l\pi}{M}
\right)
\le 0,
\qquad
0\le k,l\le M-1.
\end{equation*}

For any $v,w\in\mathcal{M}_h$, we define the discrete inner product,
the discrete $\ell_h^2$ norm, and the discrete $\ell_h^\infty$ norm by
\[
\langle v,w\rangle_h
=
h^2
\sum_{i,j=1}^{M}
v_{ij}w_{ij},
\quad
\|v\|_h
=
\sqrt{\langle v,v\rangle_h},
\quad
\|v\|_{\infty}
=
\max_{1\le i,j\le M}|v_{ij}|,
\quad
\|v\|_{H_h^1}
=
\sqrt{\|\nabla_h v\|_h^2 + \|v\|_h^2}.
\]
For vector-valued grid functions
$\boldsymbol p=(p^{(1)},p^{(2)})^T$ and
$\boldsymbol q=(q^{(1)},q^{(2)})^T$, the inner product and norm are understood
componentwise:
\[
\langle\boldsymbol p,\boldsymbol q\rangle_h
:=\sum_{m=1}^2\langle p^{(m)},q^{(m)}\rangle_h,
\qquad
\|\boldsymbol p\|_h:=\sqrt{\langle\boldsymbol p,\boldsymbol p\rangle_h}.
\]
By periodicity, the following summation-by-parts formula holds:
\begin{equation}
\langle v,\Delta_h w\rangle_h
=
-\langle\nabla_h v,\nabla_h w\rangle_h
=
\langle\Delta_h v,w\rangle_h,
\qquad
\forall\,v,w\in\mathcal{M}_h.
\label{eq_discrete_sbp}
\end{equation}

Since $f\in C^1([-\beta,\beta])$, the following standard estimate
\cite{du2021} will be used repeatedly.

\begin{lemma}\label{lem:stabilized_reaction_bound}
Under condition \eqref{eq_intro_mbp_condition}, if
$\kappa\ge \|f'\|_{L^\infty([-\beta,\beta])}$, then
\[
|f(\xi)+\kappa\xi|\le \kappa\beta,
\qquad
\xi\in[-\beta,\beta].
\]
\end{lemma}

Since $\mathcal{M}_h$ is a finite-dimensional linear space, any grid function in
$\mathcal{M}_h$ and any linear operator from $\mathcal{M}_h$ to $\mathcal{M}_h$
can be treated as a vector in $\mathbb{R}^{M^2}$ and a matrix in
$\mathbb{R}^{M^2\times M^2}$, respectively. For functions of a matrix/operator,
we have the following lemma (see \cite{du2021}).
For a linear operator $A:\mathcal{M}_h\to\mathcal{M}_h$, we use the
triple-bar notation
\[
\vertiii{A}_{h}
:=\sup_{0\ne v_h\in\mathcal{M}_h}
\frac{\|Av_h\|_h}{\|v_h\|_h}
\]
for its induced discrete $\ell_h^2$ operator norm. If $A$ is self-adjoint and
positive definite, we also write
\[
\|v_h\|_A^2:=\langle Av_h,v_h\rangle_h.
\]

\begin{lemma}\label{lem:matrix_function_calculus}
Let $\phi$ be defined on the spectrum of a diagonalizable matrix
$A\in\mathbb R^{m\times m}$, with eigenvalues
$\{\lambda_i\}_{i=1}^m$. Then:
\begin{enumerate}
    \item[(i)] $\phi(A)$ commutes with $A$, and
    $\phi(A^T)=\phi(A)^T$;

    \item[(ii)] the eigenvalues of $\phi(A)$ are
    $\{\phi(\lambda_i)\mid 1\leq i\leq m\}$;

    \item[(iii)] $\phi(P^{-1}AP)=P^{-1}\phi(A)P$ for any nonsingular
    $P\in\mathbb R^{m\times m}$.
\end{enumerate}
\end{lemma}

As a matrix, $\Delta_h$ is symmetric, negative semidefinite,
and weakly diagonally dominant with negative diagonal entries. Moreover,
$\Delta_h$ generates a contraction semigroup in the maximum norm. Hence,
the following estimate holds; see \cite{du2021,nicholas2008functions}.

\begin{lemma}\label{lem:discrete_semigroup_estimates}
For any $a,b\ge0$ and any grid function $v_h\in\mathcal{M}_h$, we have \cite{du2019}
\begin{align*}
\|e^{a\Delta_h-bI} v_h\|_{\infty}&\le e^{-b}\|v_h\|_{\infty},
&
\|e^{a\Delta_h-bI} v_h\|_h&\le e^{-b}\|v_h\|_h.
\end{align*}
If $a>0$, we also have \cite{pazy2012semigroups,hu2026energystabilityerrorestimates}
\begin{equation*}
\|e^{a\Delta_h-bI}\Delta_h v_h\|_h
\le
C e^{-b}a^{-1/2}\|\nabla_h v_h\|_h,
\end{equation*}
where $I$ denotes the identity grid operator and $C>0$ is independent of $a$, $b$, and $h$.
\end{lemma}
\begin{proof}
The first estimate follows from the $\ell^\infty$-contractivity of the
discrete heat semigroup. The second follows from the spectral bound
$e^{-a\lambda-b}\le e^{-b}$ for every eigenvalue $\lambda\ge0$ of
$-\Delta_h$.

For the third estimate, let $\lambda\ge0$ be an eigenvalue of
$-\Delta_h$. The corresponding spectral multiplier satisfies
\[
\lambda e^{-a\lambda-b}
\le
C e^{-b}a^{-1/2}\lambda^{1/2},
\]
where we have used the fact that
$\sup_{x\ge0}x^{1/2}e^{-x}<\infty$.

Therefore, the spectral calculus for the symmetric operator $-\Delta_h$ and
the identity $\|(-\Delta_h)^{1/2}v_h\|_h=\|\nabla_h v_h\|_h$ yield
\[
\|e^{a\Delta_h-bI}\Delta_h v_h\|_h
\le
C e^{-b}a^{-1/2}\|\nabla_h v_h\|_h.
\]
\end{proof}
\begin{remark}
Besides the central-difference discretization considered above, the
lumped-mass finite element method with piecewise linear basis functions
can also be used \cite{du2021}.
\end{remark}
\subsection{Generalized SAV--exponential integrator schemes}
Set
\begin{equation}\label{eq_bulk_energy_definitions}
E_1(v):=\int_\Omega F(v)\,\mathrm d\boldsymbol{x},
\qquad
E_{1h}(v_h):=\langle F(v_h),1\rangle_h,
\end{equation}
and
 introduce the invariant set
\[
\mathcal K_h^\beta
:=\{v_h\in\mathcal{M}_h:\|v_h\|_\infty\le\beta\}.
\]
Since $F$ is continuous on $[-\beta,\beta]$, there exists a constant
$C_E\ge0$, independent of $h$, such that
\begin{equation}\label{eq_bulk_energy_bounds}
|E_{1h}(v_h)|\le C_E,
\qquad v_h\in\mathcal K_h^\beta.
\end{equation}

Following the GSAV construction in \cite{ju2022generalized}, let
$\sigma:\mathbb R\to\mathbb R$ satisfy
\begin{equation}\label{eq_sigma_conditions}
\begin{aligned}
(\Sigma_1)\quad &\sigma(r)>0,
&& r\in\mathbb R,\\
(\Sigma_2)\quad &\sigma\in C^1(\mathbb R),\qquad \sigma'(r)\ge0,
&& r\in\mathbb R.
\end{aligned}
\end{equation}
For a continuous function $v$ satisfying the periodic boundary conditions and
for a periodic grid function $v_h\in\mathcal{M}_h$, define, respectively,
\begin{equation}\label{eq_ac_g_definition}
g(v,r):=\frac{\sigma(r)}{\sigma(E_1(v))}>0,
\qquad
g_h(v_h,r):=\frac{\sigma(r)}{\sigma(E_{1h}(v_h))}>0,
\qquad r\in\mathbb R.
\end{equation}

Introducing the auxiliary variable $s(t):=E_1(u(t))$ and rewriting
\eqref{eq:allen_cahn_model}, we obtain the following equivalent GSAV reformulation
\begin{align}
u_t=&\varepsilon^2\Delta u+g(u,s)f(u),\label{eq_ac_continuous_gsav}
\\
s_t=&-g(u,s)\bigl(f(u),u_t\bigr),\label{eq_ac_continuous_gsav1}
\end{align}
which is equivalent to the original Allen--Cahn equation
\eqref{eq:allen_cahn_model}.
Here and below $(\cdot,\cdot)$ denotes the continuous $L^2(\Omega)$ inner
product.
Applying the spatial discretization to \eqref{eq_ac_continuous_gsav}-\eqref{eq_ac_continuous_gsav1} yields the following semidiscrete system:
\begin{subequations}\label{eq_ac_semidiscrete_gsav}
\begin{align}
\frac{\mathrm d u_h}{\mathrm dt}
&=\varepsilon^2\Delta_h u_h
+g_h(u_h,s_h)f(u_h),
\label{eq_ac_semidiscrete_gsav_u}\\
\frac{\mathrm d s_h}{\mathrm dt}
&=-g_h(u_h,s_h)
\left\langle f(u_h),\frac{\mathrm d u_h}{\mathrm dt}\right\rangle_h.
\label{eq_ac_semidiscrete_gsav_s}
\end{align}
\end{subequations}
The associated discrete modified energy is
\begin{equation}\label{energy_discrete}
\mathcal E_h(v_h,r)
:=\frac{\varepsilon^2}{2}\|\nabla_h v_h\|_h^2+r,
\qquad v_h\in\mathcal{M}_h,\quad r\in\mathbb R.
\end{equation}

\subsection{Stabilized GSAV--ETD2 scheme}
For a fixed final time $T>0$, let $\tau>0$ be the time step size, set
$t_n=n\tau$ and $N_T:=\lfloor T/\tau\rfloor$. We initialize the auxiliary variable consistently by
\begin{equation}\label{eq_ac_initial_data}
u_h^0=\mathcal I_h u(0),
\qquad
s_h^0=E_{1h}(u_h^0),
\end{equation}
where $\mathcal I_h$ is the standard grid-restriction operator onto $\mathcal{M}_h$.
For any grid function or scalar $\phi$, introduce the time-increment notation
\begin{align}
\delta \phi^{*}
&:=\phi^{*}-\phi^n,
&
\delta \phi^{n+1}
&:=\phi^{n+1}-\phi^n,
\label{eq_time_increments}\\
\delta_\tau \phi^{*}
&:=\frac{\phi^{*}-\phi^n}{\tau},
&
\delta_\tau \phi^{n+1}
&:=\frac{\phi^{n+1}-\phi^n}{\tau}.
\label{eq_time_difference_quotients}
\end{align}

For $z\ge0$, we define the functions
\ba
\varphi_1(z)=\frac{1-e^{-z}}{z},
\quad
\varphi_2(z)=\frac{e^{-z}-1+z}{z^2},
\qquad
\bar\varphi_j(z)=\frac1{\varphi_j(z)},\quad j=1,2.
\ed
The values at $z=0$ are defined by continuity, so that
$\varphi_1(0)=1$ and $\varphi_2(0)=1/2$. Moreover,
$\varphi_j(z)>0$ for $z\ge0$; hence the reciprocal functions are well
defined on the spectra of the positive operators used below.

For a fixed stabilization parameter $\kappa>0$, we define at each time level $t_n$
\[
g_h^n:=g_h(u_h^n,s_h^n),\qquad
L_h^n:=-\varepsilon^2\Delta_h+\kappa g_h^n I,\qquad
N_h^n:=N_h(u_h^n,s_h^n)
=g_h^n f(u_h^n)+\kappa g_h^n u_h^n,
\]
where
$L_h^n$ is symmetric positive definite.  Adding and subtracting
$\kappa g_h^n u_h$ in \eqref{eq_ac_semidiscrete_gsav_u}, applying the
variation-of-constants formula on $[t_n,t_{n+1}]$, and evaluating the
nonlinear term at the left endpoint, we obtain the first-order predictor $(u_h^{*},s_h^{*})$ by
\begin{align}
u_h^{*}&=
e^{-\tau L_h^n}u_h^n
+
\tau\varphi_1(\tau L_h^n)N_h^n,
\label{eq_etd1-u}\\
s_h^{*}&=
s_h^n
-
\left\langle
g_h^n f(u_h^n),\delta u_h^{*}
\right\rangle_h.
\label{eq_etd1-s}
\end{align}
Within a fixed step from $t_n$ to $t_{n+1}$, the superscripts $*$ and
$\sharp$ denote the predictor and virtual predictor, respectively, and their
step index $n$ is suppressed.

Based on $(u_h^{*},s_h^{*})$, we define the auxiliary quantities
\[
g_h^{*}:=g_h(u_h^{*},s_h^{*}),\qquad
\bar g_h^n:=\max\{g_h^n,g_h^{*}\}.
\]
For each fixed time step, define
\[
\bar L_h^n:=\kappa\bar g_h^n I-\varepsilon^2\Delta_h,
\qquad
\bar N_h(v_h,r):=g_h(v_h,r)f(v_h)+\kappa\bar g_h^n v_h,
\quad v_h\in\mathcal{M}_h,\quad r\in\mathbb R.
\]
Once the predictor has been computed, $\bar g_h^n$, $\bar L_h^n$, and
$\bar N_h$ remain fixed throughout the corrector. The dependence of
$\bar N_h$ on the current time level through $\bar g_h^n$ is suppressed in
its notation.
Adding and subtracting $\kappa\bar g_h^n u_h$ in
\eqref{eq_ac_semidiscrete_gsav_u} and using the variation-of-constants
formula on $[t_n,t_{n+1}]$ gives
\ba
u_h(t_{n+1})
=
e^{-\tau \bar L_h^n}u_h(t_n)
+
\int_0^\tau
e^{-(\tau-\theta)\bar L_h^n}
\bar N_h
\bigl(
u_h(t_n+\theta),s_h(t_n+\theta)
\bigr)
\,\mathrm d\theta,
\ed
We approximate the integral above by the second-order ETD Runge--Kutta
formula, using $(u_h^{*},s_h^{*})$ as the predictor stage. Combining this
approximation with the second-order update of the scalar auxiliary variable
gives the following GSAV--ETD2 scheme:
\begin{align}
u_h^{n+1}&=
e^{-\tau \bar L_h^n}u_h^n
+
\tau\varphi_1(\tau \bar L_h^n)\bar N_h^n
+
\tau\varphi_2(\tau \bar L_h^n)
\left[
\bar N_h^{*}
-
\bar N_h^n
\right],
\label{eq_etd2-u}\\
s_h^{n+1}
&={}
s_h^n
-
\frac12
\left\langle
g_h^n f(u_h^n)+g_h^{*} f(u_h^{*}),\delta u_h^{n+1}
\right\rangle_h\no\\
&
-\frac34
\left\langle
\bar L_h^n (u_h^{n+1}-u_h^{\sharp}),u_h^{n+1}-u_h^{\sharp}
\right\rangle_h
-\frac47\kappa\bar g_h^n
\|u_h^{*}-u_h^{\sharp}\|_h^2,
\label{eq_s-correction}
\end{align}
where
\begin{align}
\bar N_h^n:=\bar N_h(u_h^n,s_h^n),\quad
\bar N_h^{*}:=\bar N_h(u_h^{*},s_h^{*}),\quad
u_h^{\sharp}:=
e^{-\tau \bar L_h^n}u_h^n
+\tau\varphi_1(\tau \bar L_h^n)
\bar N_h^n.
\label{eq_auxiliary_predictor1}
\end{align}
Consequently, one time step is performed sequentially: first compute the
predictor $(u_h^{*},s_h^{*})$, then form $\bar g_h^n$ and $\bar L_h^n$,
next evaluate $u_h^{n+1}$ from \eqref{eq_etd2-u}, and finally update
$s_h^{n+1}$ by \eqref{eq_s-correction}.  Every quantity on the right-hand
side is therefore known when it is used, and no nonlinear solve is required.
\begin{remark}
The two additional terms in the update of the auxiliary variable serve
different purposes. The first additional term,
$-(3/4)\langle \bar L_h^n (u_h^{n+1}-u_h^{\sharp}),u_h^{n+1}-u_h^{\sharp}\rangle_h$,
is introduced to compensate for the energy contribution arising from the
second-order correction and thereby close the discrete energy balance. In
contrast, the second additional term,
$-(4/7)\kappa\bar g_h^n\|u_h^{*}-u_h^{\sharp}\|_h^2$,
accounts for the mismatch between the stabilization factors used in the
predictor and corrector, namely, $g_h^n$ and $\bar g_h^n$, respectively.

Moreover, $u_h^{\sharp}$ is assembled from quantities already evaluated in
the second-order correction. Thus, its introduction requires no additional
matrix-function action or linear solve.

For a fixed spatial mesh and sufficiently smooth solutions, the stage
differences $u_h^{n+1}-u_h^{\sharp}$ and
$u_h^{*}-u_h^{\sharp}$ are formally $O(\tau^2)$, so the two added scalar
corrections are higher-order terms.
The mesh-uniform estimate actually needed in the convergence proof is given
later in Lemma~\ref{lem:ac_correction_defect}; that result shows directly
that these corrections do not reduce the second-order temporal accuracy.
\end{remark}
\section{Modified energy stability}\label{section3}
In this section, we establish unconditional modified energy dissipation for the GSAV--ETD2 scheme.
We first introduce the operator weights
\begin{equation}\label{eq_psi_functions}
\begin{aligned}
\psi(z)&:=\bar\varphi_1(z)-z=\frac{z}{e^z-1},
&\qquad
\psi_1(z)&:=\psi(z)+\frac z2
=\frac z2\coth\left(\frac z2\right),\\
\psi_2(z)&:=\bar\varphi_2(z)-z
=\frac{\varphi_1(z)}{\varphi_2(z)}.
\end{aligned}
\end{equation}
The values at $z=0$ are defined by continuity, namely,
$\varphi_1(0)=\psi(0)=\psi_1(0)=1$ and
$\varphi_2(0)=1/2$, $\psi_2(0)=2$.

\begin{lemma}\label{lem_phi_operator_bounds}
For any $z\ge0$, the following inequalities hold \cite{du2019,hu2026energystabilityerrorestimates}:
\begin{align}
&0<\varphi_2(z)\le\varphi_1(z)\le2\varphi_2(z)\le1,
\qquad
\frac12\le\frac{\varphi_2(z)}{\varphi_1(z)}\le1,
\label{eq_phi_scalar_bounds}\\
&0\le z\varphi_2(z)\le1,
\qquad
0\le z\varphi_2(z)^2\le\frac12,
\label{eq_phi_weighted_bounds}\\
&\bar\varphi_1(z)\le\bar\varphi_2(z)\le2\bar\varphi_1(z),
\qquad
\frac12\bar\varphi_2(z)-\bar\varphi_1(z)+\frac z2\ge0.
\label{eq_phi_reciprocal_bounds}
\end{align}
Consequently, for any self-adjoint positive semidefinite operator $A$ and every scalar grid function $v_h \in \mathcal{M}_h$,
spectral calculus yields
\begin{align}
\|\varphi_1(\tau A)v_h\|_h
&\le
\|v_h\|_h,
\qquad
\left\|
\frac{\varphi_2(\tau A)}
{\varphi_1(\tau A)}v_h
\right\|_h
\le
\|v_h\|_h,
\\
\left\|
\left(
I-
\frac{\varphi_2(\tau A)}
{\varphi_1(\tau A)}
\right)v_h
\right\|_h
&\le
\frac12\|v_h\|_h,
\qquad
\left\|
(\tau A)^{1/2}\varphi_2(\tau A)v_h
\right\|_h^2
\le
\frac12\|v_h\|_h^2,
\label{eq_phi_operator_bounds}\\
\|\psi(\tau A)v_h\|_h^2&
\le
\left\langle\psi(\tau A)v_h,v_h\right\rangle_h
\le
\|v_h\|_h^2
\le
\left\langle\psi_1(\tau A)v_h,v_h\right\rangle_h,
\label{eq_psi_norm_bounds}
\end{align}
\end{lemma}
We next establish the discrete energy dissipation property of the proposed
second-order scheme. A key difficulty arises from the fact that the predictor
and corrector employ different stabilization factors.
To resolve this mismatch, we introduce an auxiliary intermediate state
$(u_h^{\sharp},s_h^{\sharp})$, obtained by recomputing the first-order
ETD update \eqref{eq_etd1-u}--\eqref{eq_etd1-s} with the factor
$g_h^n$ replaced throughout by $\bar g_h^n$. This auxiliary
step provides an energy-consistent bridge between the solution at $t_n$ and
that at $t_{n+1}$. Accordingly, the one-step energy variation is decomposed as
\[
\mathcal E_h(u_h^{n+1},s_h^{n+1})
-\mathcal E_h(u_h^n,s_h^n)
=
\Bigl[
\mathcal E_h(u_h^{n+1},s_h^{n+1})
-\mathcal E_h(u_h^{\sharp},s_h^{\sharp})
\Bigr]
+
\Bigl[
\mathcal E_h(u_h^{\sharp},s_h^{\sharp})
-\mathcal E_h(u_h^n,s_h^n)
\Bigr].
\]
The first term describes the energy variation of the correction step, while
the second corresponds to the auxiliary predictor constructed with the same
stabilization factor. The difference between the actual predictor $u_h^{*}$
and the auxiliary state $u_h^{\sharp}$ is then controlled by the additional
dissipative correction in the scalar auxiliary-variable update. This
construction allows the two stages to be linked consistently and leads to
the unconditional discrete energy dissipation law.

\begin{theorem}\label{thm:modified_energy_stability}
For any $\tau>0$ and $\kappa>0$, the numerical solutions $(u_h^n,s_h^n)$
of the GSAV--ETD2 scheme \eqref{eq_etd1-u}--\eqref{eq_s-correction} satisfy
\begin{equation}\label{eq_ac-energy-stability}
\mathcal E_h(u_h^{n+1},s_h^{n+1})
\le
\mathcal E_h(u_h^n,s_h^n),\qquad 0\le n<N_T.
\end{equation}
\end{theorem}
\begin{proof}
We first consider the predictor step.
By polarization and summation by parts, the energy difference satisfies
\begin{align}
\mathcal E_h(u_h^{*},s_h^{*})
-\mathcal E_h(u_h^n,s_h^n)
={}&
\frac{\varepsilon^2}{2}
\left(
\|\nabla_h u_h^{*}\|_h^2
-\|\nabla_h u_h^n\|_h^2
\right)
+s_h^{*}-s_h^n
\nonumber\\
={}&
-\frac{\varepsilon^2}{2}
\|\nabla_h \delta u_h^{*}\|_h^2
-\varepsilon^2
\left\langle
\Delta_h u_h^{*},
\delta u_h^{*}
\right\rangle_h
+s_h^{*}-s_h^n
\nonumber\\
={}&
-\frac{\varepsilon^2}{2}
\|\nabla_h \delta u_h^{*}\|_h^2
+
\left\langle
L_h^n u_h^{*},
\delta u_h^{*}
\right\rangle_h
-\kappa g_h^n\left\langle
u_h^{*},
\delta u_h^{*}
\right\rangle_h
+s_h^{*}-s_h^n.
\label{eq_predictor_energy_difference}
\end{align}
Applying the operator
$\tau^{-1}(\psi(\tau L_h^n)+\tau L_h^n)$ to
\eqref{eq_etd1-u}, we can rewrite the predictor as the equivalent backward-Euler form
\begin{equation}\label{eq_etd1-predictor-be}
\psi(\tau L_h^n)\delta_\tau u_h^{*}
+
L_h^n u_h^{*}
=
N_h^n.
\end{equation}
Taking the discrete inner product of
\eqref{eq_etd1-predictor-be} with $\delta u_h^{*}$  gives
\begin{align}
\langle L_h^n u_h^{*},\delta u_h^{*}\rangle_h
={}&-\frac1\tau
\left\langle
\psi(\tau L_h^n)\delta u_h^{*},
\delta u_h^{*}
\right\rangle_h
+\langle g_h^n f(u_h^n),\delta u_h^{*}\rangle_h
+\kappa g_h^n\langle u_h^n,\delta u_h^{*}\rangle_h .
\label{eq_predictor_test_identity}
\end{align}
Combining \eqref{eq_predictor_energy_difference},
\eqref{eq_predictor_test_identity}, and the scalar predictor
\eqref{eq_etd1-s}, we obtain
\begin{align}
&\mathcal E_h(u_h^{*},s_h^{*})-\mathcal E_h(u_h^n,s_h^n)\no\\
={}&-\frac1\tau
\left\langle
\psi(\tau L_h^n)\delta u_h^{*},
\delta u_h^{*}
\right\rangle_h
-\frac12\left\langle
L_h^n\delta u_h^{*},
\delta u_h^{*}
\right\rangle_h
-\frac{\kappa g_h^n}{2}\|\delta u_h^{*}\|_h^2
\no\\
={}&
-\frac1\tau
\left\langle
\psi_1(\tau L_h^n)\delta u_h^{*},
\delta u_h^{*}
\right\rangle_h
-\frac{\kappa g_h^n}{2}\|\delta u_h^{*}\|_h^2
\le0,
\label{eq_predictor_energy_decay}
\end{align}
where the last inequality follows from Lemma~\ref{lem_phi_operator_bounds} and
$g_h^n>0$, which follows from \eqref{eq_sigma_conditions} and
\eqref{eq_ac_g_definition}. Thus, the predictor satisfies the energy estimate
\eqref{eq_predictor_energy_decay}, which will be used below.

We next turn to the corrected solution.
Recalling \eqref{eq_auxiliary_predictor1} and \eqref{eq_etd2-u},
we define the auxiliary predictor $(u_h^{\sharp},s_h^{\sharp})$
by employing the stabilization factor $\bar g_h^n$ as follows:
\begin{align}
u_h^{\sharp}
&:=e^{-\tau \bar L_h^n}u_h^n
+\tau\varphi_1(\tau \bar L_h^n)\bar N_h^n,\label{eq_auxiliary_predictor2}
\\
s_h^{\sharp}
&:=s_h^n-\left\langle
\bar g_h^n f(u_h^n),u_h^{\sharp}-u_h^n
\right\rangle_h .
\end{align}

For each fixed time step $n$, introduce the auxiliary increments
\begin{align}
\bigl(v_h,w_h,d_h\bigr)
&:=
\bigl(
u_h^{\sharp}-u_h^n,
u_h^{n+1}-u_h^{\sharp},
u_h^{*}-u_h^{\sharp}
\bigr).
\label{eq_virtual_time_increments}
\end{align}
In particular,
\begin{equation}
\delta u_h^{n+1}=v_h+w_h,
\qquad
\delta u_h^{*}=v_h+d_h,
\qquad
u_h^{*}-u_h^{n+1}=d_h-w_h.
\label{eq_virtual_increment_relations}
\end{equation}
Applying $\tau^{-1}\bar\varphi_1(\tau \bar L_h^n)$ to the auxiliary
predictor \eqref{eq_auxiliary_predictor2} gives
\begin{equation}
\psi(\tau \bar L_h^n)\tau^{-1}v_h
+
\bar L_h^n u_h^{\sharp}
=
\bar N_h^n
=\bar g_h^n f(u_h^n)+\kappa\bar g_h^n u_h^n.
\label{eq_auxiliary_predictor_reformulated}
\end{equation}
Similarly to \eqref{eq_predictor_energy_difference}, the energy difference
between the auxiliary stage and the previous step satisfies
\begin{align}
\mathcal E_h(u_h^{\sharp},s_h^{\sharp})-\mathcal E_h(u_h^n,s_h^n)
={}-\frac1\tau
\left\langle
\psi_1(\tau \bar L_h^n)v_h,
v_h
\right\rangle_h
-\frac{\kappa\bar g_h^n}{2}
\|v_h\|_h^2
\le0.
\label{eq_auxiliary_predictor_decay}
\end{align}

We now evaluate the energy change from the auxiliary stage to the
corrected solution.  Polarization and summation by parts give
\begin{align}
&\mathcal E_h(u_h^{n+1},s_h^{n+1})
-\mathcal E_h(u_h^{\sharp},s_h^{\sharp})
\nonumber\\
={}&-\frac{\varepsilon^2}{2}
\|\nabla_h w_h\|_h^2
-\varepsilon^2
\left\langle
\Delta_h u_h^{n+1},w_h
\right\rangle_h
+s_h^{n+1}-s_h^{\sharp}
\nonumber\\
={}&-\frac{\varepsilon^2}{2}
\|\nabla_h w_h\|_h^2
+
\left\langle
\bar L_h^n u_h^{n+1},
w_h
\right\rangle_h
-\kappa \bar g_h^n\left\langle
u_h^{n+1},
 w_h
\right\rangle_h
+s_h^{n+1}-s_h^{\sharp}.
\label{eq_corrector_energy_from_auxiliary}
\end{align}
Applying $\tau^{-1}\bar\varphi_2(\tau \bar L_h^n)$ to the corrector \eqref{eq_etd2-u} gives
\begin{equation}
\bar\varphi_2(\tau \bar L_h^n)
\tau^{-1}w_h
=\bar N_h^{*}-\bar N_h^n.
\label{eq_corrector_increment_identity}
\end{equation}
Adding \eqref{eq_auxiliary_predictor_reformulated} to
\eqref{eq_corrector_increment_identity}, and using
$\psi_2(z)=\bar\varphi_2(z)-z$, yields
\begin{equation}
\psi_2(\tau \bar L_h^n)
\tau^{-1}w_h
+\psi(\tau \bar L_h^n)
\tau^{-1}v_h
+\bar L_h^n u_h^{n+1}
=g_h^{*} f(u_h^{*})+\kappa\bar g_h^n u_h^{*}.
\label{eq_corrector_reformulated}
\end{equation}
Taking the discrete inner product of
\eqref{eq_corrector_reformulated} with
$w_h$ yields
\begin{align}
\left\langle
\bar L_h^n u_h^{n+1},w_h
\right\rangle_h
=&-
\left\langle
\psi_2(\tau \bar L_h^n)
\tau^{-1}w_h,
w_h
\right\rangle_h
-
\left\langle
\psi(\tau \bar L_h^n)
\tau^{-1}v_h,
w_h
\right\rangle_h\no\\
&
+\left\langle
g_h^{*} f(u_h^{*}),w_h
\right\rangle_h+\kappa\bar g_h^n
\left\langle
u_h^{*},w_h
\right\rangle_h
.
\label{eq_corrector_equation_test}
\end{align}

Equation \eqref{eq_corrector_increment_identity} also gives the exact
nonlinear increment identity
\begin{equation}
g_h^{*} f(u_h^{*})-g_h^n f(u_h^n)
=\bar\varphi_2(\tau \bar L_h^n)
 \tau^{-1}w_h
-\kappa\bar g_h^n\delta u_h^{*}.
\label{eq_nonlinear_increment_identity}
\end{equation}
In particular, the sign of the stabilization contribution in
\eqref{eq_nonlinear_increment_identity} is negative.
On the other hand, the definitions of $s_h^{\sharp}$ and $s_h^{n+1}$ imply
\begin{align}
&s_h^{n+1}-s_h^{\sharp}
+\left\langle g_h^{*} f(u_h^{*}),w_h\right\rangle_h
\nonumber\\
={}&-\frac12
\left\langle g_h^n f(u_h^n)+g_h^{*} f(u_h^{*}),\delta u_h^{n+1}\right\rangle_h
+\left\langle g_h^n f(u_h^n),v_h\right\rangle_h
+\left\langle g_h^{*} f(u_h^{*}),w_h\right\rangle_h
\nonumber\\
&-\frac34
\left\langle
\bar L_h^n w_h,
w_h
\right\rangle_h
-\frac47\kappa\bar g_h^n
\|d_h\|_h^2
\nonumber\\
={}&\frac12
\left\langle
g_h^{*} f(u_h^{*})-g_h^n f(u_h^n),w_h
\right\rangle_h
-\frac12
\left\langle
g_h^{*} f(u_h^{*})-g_h^n f(u_h^n),v_h
\right\rangle_h
\nonumber\\
&-\frac34
\left\langle
\bar L_h^n w_h,
w_h
\right\rangle_h
-\frac47\kappa\bar g_h^n
\|d_h\|_h^2
\nonumber\\
={}&\frac12
\left\langle
\bar\varphi_2(\tau \bar L_h^n)
\tau^{-1}w_h
-\kappa\bar g_h^n\delta u_h^{*},
w_h
\right\rangle_h-\frac12
\left\langle
\bar\varphi_2(\tau \bar L_h^n)
\tau^{-1}w_h
-\kappa\bar g_h^n\delta u_h^{*},
v_h
\right\rangle_h
\nonumber\\
&-\frac34
\left\langle
\bar L_h^n w_h,
w_h
\right\rangle_h
-\frac47\kappa\bar g_h^n
\|d_h\|_h^2.
\label{eq_scalar_correction_detailed}
\end{align}
The last equality uses \eqref{eq_nonlinear_increment_identity}; in
particular, both stabilization terms have the signs dictated by that exact
identity.

Combining \eqref{eq_corrector_energy_from_auxiliary},
\eqref{eq_corrector_equation_test}, and
\eqref{eq_scalar_correction_detailed}, and using
\eqref{eq_virtual_increment_relations}, we obtain
\begin{align}
&\mathcal E_h(u_h^{n+1},s_h^{n+1})
-\mathcal E_h(u_h^{\sharp},s_h^{\sharp})
\nonumber\\
={}&-\frac{\varepsilon^2}{2}
\|\nabla_h w_h\|_h^2
-
\left\langle
\psi_2(\tau \bar L_h^n)
\tau^{-1}w_h,
w_h
\right\rangle_h
\nonumber\\
&-
\left\langle
\psi(\tau \bar L_h^n)
\tau^{-1}v_h,
w_h
\right\rangle_h
+\kappa\bar g_h^n
\left\langle
d_h
-w_h,
w_h
\right\rangle_h
\nonumber\\
&+\frac12
\left\langle
\bar\varphi_2(\tau \bar L_h^n)
\tau^{-1}w_h
-\kappa\bar g_h^n(v_h+d_h),
w_h
\right\rangle_h
\nonumber\\
&-\frac12
\left\langle
\bar\varphi_2(\tau \bar L_h^n)
\tau^{-1}w_h
-\kappa\bar g_h^n(v_h+d_h),
v_h
\right\rangle_h
\nonumber\\
&-\frac34
\left\langle
\bar L_h^n w_h,
w_h
\right\rangle_h
-\frac47\kappa\bar g_h^n
\|d_h\|_h^2.
\label{eq_corrector_energy_detailed}
\end{align}
Adding \eqref{eq_auxiliary_predictor_decay} to
\eqref{eq_corrector_energy_detailed}, using
\eqref{eq_virtual_increment_relations}, and collecting like terms gives the
following exact identity.  To make the coefficient of the pure $w_h$ term
explicit, we use
$-\varepsilon^2\Delta_h=\bar L_h^n-\kappa\bar g_h^n I$ and
$\psi_2(\tau\bar L_h^n)=\bar\varphi_2(\tau\bar L_h^n)
-\tau\bar L_h^n$ to obtain
\begin{align*}
&-\frac{\varepsilon^2}{2}\|\nabla_h w_h\|_h^2
-\frac1\tau
\left\langle\psi_2(\tau\bar L_h^n)w_h,w_h\right\rangle_h
-\kappa\bar g_h^n\|w_h\|_h^2
\\
&\qquad
+\frac1{2\tau}
\left\langle\bar\varphi_2(\tau\bar L_h^n)w_h,w_h\right\rangle_h
-\frac34\left\langle\bar L_h^n w_h,w_h\right\rangle_h
\\
&=-\frac1\tau
\left\langle
\left(
\frac12\bar\varphi_2(\tau\bar L_h^n)
+\frac{\tau\kappa\bar g_h^n}{2}I
+\frac\tau4\bar L_h^n
\right)w_h,w_h
\right\rangle_h.
\end{align*}
Substituting this into \eqref{eq_corrector_energy_detailed} gives
\begin{align}
&\mathcal E_h(u_h^{n+1},s_h^{n+1})-\mathcal E_h(u_h^n,s_h^n)
\nonumber\\
={}&-\frac1\tau\Bigg\{
\left\langle
\psi_1(\tau \bar L_h^n)
v_h,
v_h
\right\rangle_h
+\frac12\left\langle
\left(
2\psi(\tau \bar L_h^n)
+\bar\varphi_2(\tau \bar L_h^n)
+\tau\kappa\bar g_h^n I
\right)
v_h,
w_h
\right\rangle_h
\nonumber\\
&+\left\langle
\left(
\frac12\bar\varphi_2(\tau \bar L_h^n)
+\frac{\tau\kappa\bar g_h^n}{2}I
+\frac\tau4\bar L_h^n
\right)
w_h,
w_h
\right\rangle_h-\frac{\tau\kappa\bar g_h^n}{2}
\left\langle
v_h,
d_h
\right\rangle_h
\nonumber\\
&-\frac{\tau\kappa\bar g_h^n}{2}
\left\langle
w_h,
d_h
\right\rangle_h
+\frac{4\tau\kappa\bar g_h^n}{7}
\|d_h\|_h^2
\Bigg\}.
\label{eq_corrector_operator_energy_identity}
\end{align}

We next determine the sign of the right-hand side of
\eqref{eq_corrector_operator_energy_identity}.
Since $\bar L_h^n$ is symmetric positive definite, the spectral theorem yields an $h$-orthonormal eigenbasis $\{e_j\}_j$ such that
\begin{align}
\bar L_h^n e_j=\lambda_j e_j,
\quad
\psi(\tau \bar L_h^n)e_j
=\psi(\tau\lambda_j)e_j,
\quad
\bar\varphi_2(\tau \bar L_h^n)e_j
=\bar\varphi_2(\tau\lambda_j)e_j,\quad \langle e_i,e_j\rangle_h=\delta_{ij}.
\label{eq_spectral_representation}
\end{align}
Here, the last two identities follow from the fact that
$\psi(\tau \bar L_h^n)$ and $\bar\varphi_2(\tau \bar L_h^n)$ are functions of $\bar L_h^n$.

For each $\xi_h\in\{v_h,w_h,d_h\}$, its spectral representation is
\ba\label{eq:modal_expansion}
\xi_h=\sum_j \xi_j e_j,
\qquad
\xi_j:=\langle \xi_h,e_j\rangle_h.
\ed
Set $\alpha:=\tau\kappa\bar g_h^n$ and $z_j:=\tau\lambda_j$.
Using \eqref{eq_spectral_representation} and
\eqref{eq:modal_expansion}, the individual terms in
\eqref{eq_corrector_operator_energy_identity} become
\begin{align*}
\left\langle
\psi_1(\tau \bar L_h^n)v_h,v_h
\right\rangle_h
&=
\sum_j
\psi_1(z_j)v_j^2,
\\
\frac12
\left\langle
\left(
2\psi(\tau \bar L_h^n)
+\bar\varphi_2(\tau \bar L_h^n)
+\alpha I
\right)v_h,w_h
\right\rangle_h
&=
\sum_j
\frac{
2\psi(z_j)+\bar\varphi_2(z_j)+\alpha
}{2}
v_jw_j,
\\
-\frac{\alpha}{2}\langle v_h,d_h\rangle_h
&=
-\sum_j\frac{\alpha}{2}v_jd_j,
\\
\left\langle
\left(
\frac12\bar\varphi_2(\tau \bar L_h^n)
+\frac{\alpha}{2}I
+\frac{\tau}{4}\bar L_h^n
\right)w_h,w_h
\right\rangle_h
&=
\sum_j
\left(
\frac12\bar\varphi_2(z_j)
+\frac{\alpha}{2}
+\frac{z_j}{4}
\right)w_j^2,
\\
-\frac{\alpha}{2}\langle w_h,d_h\rangle_h
&=
-\sum_j\frac{\alpha}{2}w_jd_j,
\\
\frac{4\alpha}{7}\|d_h\|_h^2
&=
\sum_j\frac{4\alpha}{7}d_j^2.
\end{align*}
Consequently, \eqref{eq_corrector_operator_energy_identity} can be written as
\begin{align}
\mathcal E_h(u_h^{n+1},s_h^{n+1})-\mathcal E_h(u_h^n,s_h^n)
=-\frac{1}{\tau}
\sum_j\Bigg[
&\left(\psi(z_j)+\frac{z_j}{2}\right)v_j^2
+\frac{
2\psi(z_j)+\bar\varphi_2(z_j)+\alpha
}{2}v_jw_j
-\frac{\alpha}{2}v_jd_j
\nonumber\\
&+
\left(
\frac12\bar\varphi_2(z_j)
+\frac{\alpha}{2}
+\frac{z_j}{4}
\right)w_j^2
-\frac{\alpha}{2}w_jd_j
+\frac{4\alpha}{7}d_j^2
\Bigg].
\label{eq_modal_scalar_quadratic}
\end{align}

For each mode $j$, define
$\mathbf X_j:=(v_j,w_j,d_j)^T.$
Therefore, comparison with
\eqref{eq_modal_scalar_quadratic} gives
\begin{equation}
\mathcal E_h(u_h^{n+1},s_h^{n+1})-\mathcal E_h(u_h^n,s_h^n)
=
-\frac{1}{\tau}
\sum_j
\mathbf X_j^TM(z_j,\alpha)\mathbf X_j,
\label{eq_Bh_modal_matrix}
\end{equation}
where
\begin{equation}
M(z,\alpha)
=
\begin{pmatrix}
\psi(z)+\dfrac z2
&
\dfrac{2\psi(z)+\bar\varphi_2(z)+\alpha}{4}
&
-\dfrac{\alpha}{4}
\\[2mm]
\dfrac{2\psi(z)+\bar\varphi_2(z)+\alpha}{4}
&
\dfrac{\bar\varphi_2(z)}2+\dfrac{\alpha}{2}+\dfrac z4
&
-\dfrac{\alpha}{4}
\\[2mm]
-\dfrac{\alpha}{4}
&
-\dfrac{\alpha}{4}
&
\dfrac{4\alpha}{7}
\end{pmatrix}.
\label{eq_corrector_energy_matrix}
\end{equation}

It remains to prove that
\[
M(z_j,\alpha)\succeq0
\qquad\text{for every }j.
\]
To this end, let
$-\Delta_h e_j=\mu_j e_j$ with $\mu_j\ge0$. Then, we have
\[
z_j
=
\tau\lambda_j=\tau(\kappa\bar g_h^n+\varepsilon^2\mu_j)
=
\alpha+\tau\varepsilon^2\mu_j,
\]
which implies $0<\alpha\le z_j$.  Fixing a mode $j$ and writing
$z=z_j$, the affine dependence of $M(z,\alpha)$ on $\alpha$ gives the
following convex combination of $M(z,0)$ and $M(z,z)$:
\begin{equation}
M(z,\alpha)
=
\left(1-\frac{\alpha}{z}\right)M(z,0)
+
\frac{\alpha}{z}M(z,z).
\label{eq_energy_matrix_convexity}
\end{equation}
 Consequently, it suffices to prove
that both endpoint matrices are positive semidefinite.

For $\alpha=0$, direct substitution gives
\[
M(z,0)
=
\begin{pmatrix}
\psi_1(z)
&
\dfrac12\psi_1(z)+\dfrac14\bar\varphi_2(z)-\dfrac z4
&
0
\\[2mm]
\dfrac12\psi_1(z)+\dfrac14\bar\varphi_2(z)-\dfrac z4
&
\dfrac12\bar\varphi_2(z)+\dfrac z4
&
0
\\[2mm]
0&0&0
\end{pmatrix}.
\]
By Lemma~\ref{lem_phi_operator_bounds}, we have
\ba\label{eq:phi_auxiliary_bounds}
\psi_1(z)\ge \frac z2,
\qquad
0\le
\frac12\bar\varphi_2(z)-\psi_1(z)
\le \frac z2.
\ed
The determinant of the leading $2\times2$ principal block of $M(z,0)$ is
\begin{align}
D_0
&=
\psi_1(z)
\left(
\frac12\bar\varphi_2(z)+\frac z4
\right)
-
\left(
\frac12\psi_1(z)
+\frac14\bar\varphi_2(z)
-\frac z4
\right)^2
\nonumber\\
&=
\psi_1(z)
\left[
\psi_1(z)
+
\left(
\frac12\bar\varphi_2(z)-\psi_1(z)
\right)
+\frac z4
\right]
\nonumber\\
&\quad-
\left[
\psi_1(z)
+\frac12
\left(
\frac12\bar\varphi_2(z)-\psi_1(z)
\right)
-\frac z4
\right]^2
\nonumber\\
&=
\frac{3z}{4}\psi_1(z)-\frac{z^2}{16}
+
\left(
\frac12\bar\varphi_2(z)-\psi_1(z)
\right)
\left[
\frac z4
-\frac14
\left(
\frac12\bar\varphi_2(z)-\psi_1(z)
\right)
\right]\no\\
&\ge
\frac{5z^2}{16}
\ge0,
\label{eq_M0_determinant}
\end{align}
Since $\psi_1(z)>0$ and the third row and column of $M(z,0)$ vanish,
the nonnegativity of this determinant implies
$M(z,0)\succeq0$.

We next consider the other endpoint $\alpha=z$. Direct substitution gives
\[
M(z,z)
=
\begin{pmatrix}
\psi_1(z)
&
\dfrac12\psi_1(z)+\dfrac14\bar\varphi_2(z)
&
-\dfrac z4
\\[2mm]
\dfrac12\psi_1(z)+\dfrac14\bar\varphi_2(z)
&
\dfrac12\bar\varphi_2(z)+\dfrac{3z}{4}
&
-\dfrac z4
\\[2mm]
-\dfrac z4&-\dfrac z4&\dfrac{4z}{7}
\end{pmatrix}.
\]
For $\mathbf X_j=(v_j,w_j,d_j)^T$, a direct rearrangement and completion of the square yield
\begin{align}
\mathbf X_j^T M(z,z)\mathbf X_j
={}&
\frac{4z}{7}
\left(
d_j-\frac7{16}(v_j+w_j)
\right)^2
\nonumber\\
&+
\left(
\psi_1(z)-\frac{7z}{64}
\right)(v_j+w_j)^2
\nonumber\\
&+
\left(
\frac12\bar\varphi_2(z)-\psi_1(z)
\right)(v_j+w_j)w_j
+
\frac{3z}{4}w_j^2.
\label{eq_energy_matrix_endpoint_square}
\end{align}
Hence, Young's inequality and \eqref{eq:phi_auxiliary_bounds} give
\[
\left(
\frac12\bar\varphi_2(z)-\psi_1(z)
\right)(v_j+w_j)w_j
\ge
-\frac z4
\left[
(v_j+w_j)^2+w_j^2
\right].
\]
Consequently,
\begin{align}
\mathbf X_j^T M(z,z)\mathbf X_j
&\ge
\frac{4z}{7}
\left(
d_j-\frac7{16}(v_j+w_j)
\right)^2
+
\frac{9z}{64}(v_j+w_j)^2
+
\frac z2 w_j^2
\ge0.
\end{align}
Therefore,
\begin{equation}
M(z,z)\succeq0.
\label{eq_Mzz_psd}
\end{equation}

Combining \eqref{eq_M0_determinant}, \eqref{eq_Mzz_psd}, and
\eqref{eq_energy_matrix_convexity}, we conclude that
\[
M(z,\alpha)\succeq0
\qquad
\text{for all }z>0,\quad 0\le\alpha\le z.
\]
Therefore, we can conclude that
\[
\mathcal E_h(u_h^{n+1},s_h^{n+1})
-
\mathcal E_h(u_h^n,s_h^n)
=-\frac{1}{\tau}\sum_j
\mathbf X_j^TM(z_j,\alpha)\mathbf X_j
\le0.
\]
This proves the inequality in
\eqref{eq_ac-energy-stability}.
\end{proof}

\section{Maximum bound principle and uniform bounds for the GSAV factors}
\label{section4}
In this section, we prove that the GSAV--ETD2 scheme preserves the maximum bound principle unconditionally, and we establish uniform bounds for the GSAV factors $g_h^n$ and $\bar g_h^n$.
\begin{theorem}[Maximum bound principle for the GSAV--ETD2 scheme]\label{thm:mbp-gsav-etd2}
    Assume that the
initial data satisfy $\|u_h^0\|_\infty\le\beta$.
Under \eqref{eq_intro_mbp_condition}, if
$\kappa\ge \|f'\|_{L^\infty([-\beta,\beta])}$, then the numerical solution
$(u_h^n,s_h^n)$ generated by the GSAV--ETD2 scheme
\eqref{eq_etd1-u}--\eqref{eq_s-correction} preserves the maximum bound
principle unconditionally:
\[
\|u_h^n\|_\infty\le\beta,\qquad 0\le n\le N_T.
\]
\end{theorem}

\begin{proof}
We proceed by induction. Assume that
$\|u_h^n\|_\infty\le\beta$.
We first consider the predictor \eqref{eq_etd1-u}.
By the induction hypothesis and
Lemma~\ref{lem:stabilized_reaction_bound}, we have
\ba\label{eq_N-ei1-bound}
\|N_h^n\|_\infty
\le g_h^n\|f(u_h^n)+\kappa u_h^n\|_\infty\le \kappa g_h^n\beta.
\ed
The scalar identity
$\tau\varphi_1(\tau z)=\int_0^\tau
e^{-(\tau-\theta)z}\,\mathrm d\theta$, together with
Lemma~\ref{lem:matrix_function_calculus}, gives the corresponding operator
identity for $L_h^n$.
Taking the maximum norm of the resulting integral representation of
\eqref{eq_etd1-u}, applying
Lemma~\ref{lem:discrete_semigroup_estimates}, and using
\eqref{eq_N-ei1-bound}, we obtain
\begin{align}
\|u_h^{*}\|_\infty
&\le
e^{-\tau\kappa g_h^n}\|u_h^n\|_\infty
+
\int_0^\tau
e^{-(\tau-\theta)\kappa g_h^n}
\|N_h^n\|_\infty\,\mathrm d\theta\no
\\
&\le
e^{-\tau\kappa g_h^n}\beta
+
\kappa\beta g_h^n
\int_0^\tau
e^{-(\tau-\theta)\kappa g_h^n}\,\mathrm d\theta\no
\\
&=
e^{-\tau\kappa g_h^n}\beta
+
\bigl(1-e^{-\tau\kappa g_h^n}\bigr)\beta
=
\beta.\label{eq_predictor-etd2-mbp-bound}
\end{align}

We next consider the corrector \eqref{eq_etd2-u}. Since
$0<g_h^n,g_h^{*}\le \bar g_h^n$, the induction hypothesis, together with
\eqref{eq_predictor-etd2-mbp-bound} and
Lemma~\ref{lem:stabilized_reaction_bound}, yields
\begin{align}
\|\bar N_h^n\|_\infty
&\le
g_h^n\|f(u_h^n)+\kappa u_h^n\|_\infty
+\kappa(\bar g_h^n-g_h^n)\|u_h^n\|_\infty
\le
\kappa\bar g_h^n\beta,
\nonumber\\
\|\bar N_h^{*}\|_\infty
&\le
g_h^{*}\|f(u_h^{*})+\kappa u_h^{*}\|_\infty
+\kappa(\bar g_h^n-g_h^{*})\|u_h^{*}\|_\infty
\le
\kappa\bar g_h^n\beta.
\label{eq_stage-nonlinearity-etd2-bound}
\end{align}
The corrected solution \eqref{eq_etd2-u} admits the integral representation
\begin{align}
u_h^{n+1}
={}
e^{-\tau \bar L_h^n}u_h^n
+
\int_0^\tau
e^{-(\tau-\theta)\bar L_h^n}
\left[
\left(1-\frac{\theta}{\tau}\right)
\bar N_h^n
+
\frac{\theta}{\tau}
\bar N_h^{*}
\right]\,\mathrm d\theta.
\label{eq_etd2-mbp-integral}
\end{align}
Taking the maximum norm of \eqref{eq_etd2-mbp-integral} and applying
Lemmas~\ref{lem:matrix_function_calculus}
and~\ref{lem:discrete_semigroup_estimates}, we obtain
\begin{align*}
\|u_h^{n+1}\|_\infty
&\le
e^{-\tau\kappa\bar g_h^n}\|u_h^n\|_\infty
+
\int_0^\tau
e^{-(\tau-\theta)\kappa\bar g_h^n}
\left[
\left(1-\frac{\theta}{\tau}\right)
\|\bar N_h^n\|_\infty
+
\frac{\theta}{\tau}
\|\bar N_h^{*}\|_\infty
\right]\,\mathrm d\theta
\\
&\le
e^{-\tau\kappa\bar g_h^n}\beta
+
\kappa\bar g_h^n\beta
\int_0^\tau
e^{-(\tau-\theta)\kappa\bar g_h^n}
\left[
\left(1-\frac{\theta}{\tau}\right)
+
\frac{\theta}{\tau}
\right]\,\mathrm d\theta
\\
&=
e^{-\tau\kappa\bar g_h^n}\beta
+
\kappa\bar g_h^n\beta
\int_0^\tau
e^{-(\tau-\theta)\kappa\bar g_h^n}\,\mathrm d\theta
\\
&=
e^{-\tau\kappa\bar g_h^n}\beta
+
\bigl(1-e^{-\tau\kappa\bar g_h^n}\bigr)\beta
=
\beta.
\end{align*}
Therefore, $\|u_h^{n+1}\|_\infty\le\beta$. Starting from
$\|u_h^0\|_\infty\le\beta$, induction proves the asserted bound for the
corrected solution; the predictor bound is recorded separately in
\eqref{eq_predictor-etd2-mbp-bound}.
\end{proof}

The energy estimate and the maximum bound principle together control both
arguments of the quotient defining $g_h$. This gives the following uniform
upper bound.
\begin{lemma}[Uniform upper bounds for the GSAV factors]\label{lem:g_upper_ac}
Assume the hypotheses of Theorem~\ref{thm:mbp-gsav-etd2} and suppose that
$\mathcal E_h(u_h^0,s_h^0)\le C_0$, where $C_0$ is independent of $h$.
Then there exists a constant $G^+>0$, independent of $h$ and $\tau$, such
that
\begin{align*}
0<g_h^n&\le G^+,
&&0\le n\le N_T,\\
0<g_h^{*},\bar g_h^n&\le G^+,
&&0\le n<N_T.
\end{align*}
\end{lemma}
\begin{proof}
We first estimate $g_h^n$ and $g_h^{*}$. By
Theorem~\ref{thm:modified_energy_stability} and the
predictor energy identity \eqref{eq_predictor_energy_decay}, we have
\[
\mathcal E_h(u_h^{*},s_h^{*})
\le
\mathcal E_h(u_h^n,s_h^n)
\le
\mathcal E_h(u_h^0,s_h^0)
\le C_0.
\]
By the definition of the discrete energy \eqref{energy_discrete}, we have
\begin{equation}\label{eq_ac_s_upper_bounds}
s_h^n\le C_0,\quad 0\le n\le N_T,
\qquad
s_h^{*}\le C_0,\quad 0\le n<N_T.
\end{equation}
By Theorem~\ref{thm:mbp-gsav-etd2}, the predictor bound
\eqref{eq_predictor-etd2-mbp-bound}, and the bulk-energy bounds
\eqref{eq_bulk_energy_bounds}, there exists a constant $C_E>0$, independent
of $h$ and $\tau$, such that
\[
-C_E
\le
E_{1h}(u_h^n),\,E_{1h}(u_h^{*})
\le
C_E.
\]
Therefore, the positivity and monotonicity in \eqref{eq_sigma_conditions} give
\begin{equation}
g_h^n
=
\frac{\sigma(s_h^n)}{\sigma(E_{1h}(u_h^n))}
\le
\frac{\sigma(C_0)}{\sigma(-C_E)},
\qquad
g_h^{*}
=
\frac{\sigma(s_h^{*})}{\sigma(E_{1h}(u_h^{*}))}
\le
\frac{\sigma(C_0)}{\sigma(-C_E)}.
\label{eq_uniform_g_upper_bound}
\end{equation}
Setting $G^+:=\sigma(C_0)/\sigma(-C_E)$ gives the desired uniform upper bound for $g_h^n$ and $g_h^{*}$.
The bound for $\bar g_h^n=\max\{g_h^n,g_h^{*}\}$ follows immediately.
\end{proof}
\begin{lemma}[Discrete time regularity]\label{lem:discrete_time_regularity}
Under the hypotheses of Theorem~\ref{thm:mbp-gsav-etd2} and
Lemma~\ref{lem:g_upper_ac}, impose the consistent initialization
\eqref{eq_ac_initial_data}. Then, for any $\tau>0$, there
exists a constant
$C_T>0$, independent of $h$ and $\tau$, such that
\begin{align}
\max_{0\le k\le N_T}\|u_h^k\|_{H_h^1}
+\tau\sum_{n=0}^{N_T-1}
\|\delta_\tau u_h^{n+1}\|_h^2
&\le C_T,
\label{eq_ac-uniform-bound}\\
\|u_h^{*}-u_h^{\sharp}\|_h^2
&\le C_T\tau,\qquad 0\le n<N_T.
\label{eq_predictor_auxiliary_difference_bound}
\end{align}
\end{lemma}
\begin{proof}
By Theorem~\ref{thm:mbp-gsav-etd2}, the predictor estimate
\eqref{eq_predictor-etd2-mbp-bound}, and Lemma~\ref{lem:g_upper_ac}, we obtain
\begin{equation}\label{eq_ac_coarse_nonlinear_increment_bound}
\|N_h^n\|_h+\|\bar N_h^n\|_h+\|\bar N_h^{*}\|_h\le C,
\qquad 0\le n<N_T.
\end{equation}
Applying the operator
$\frac1\tau\left(\psi(\tau \bar L_h^n)
+\tau \bar L_h^n\right)$
to \eqref{eq_etd2-u}  gives the transformed ETD2 equation
\begin{align}
\psi(\tau \bar L_h^n)\delta_\tau u_h^{n+1}
+\bar L_h^n u_h^{n+1}
={}
\left(I-\frac{\varphi_2(\tau \bar L_h^n)}
{\varphi_1(\tau \bar L_h^n)}\right)
\bar N_h^n
+\frac{\varphi_2(\tau \bar L_h^n)}
{\varphi_1(\tau \bar L_h^n)}
\bar N_h^{*}:=\mathcal R_h^n.
\label{eq_ac_rk2_regularity_equation}
\end{align}
Lemma~\ref{lem_phi_operator_bounds} and the uniform bounds \eqref{eq_ac_coarse_nonlinear_increment_bound}
 give
$\|\mathcal R_h^n\|_h\le C$.
Taking the discrete inner product of
\eqref{eq_ac_rk2_regularity_equation} with
$\delta_\tau u_h^{n+1}$ gives
\begin{align}
\left\langle
\psi(\tau \bar L_h^n)\delta_\tau u_h^{n+1},
\delta_\tau u_h^{n+1}
\right\rangle_h
+
\left\langle
\bar L_h^n u_h^{n+1},
\delta_\tau u_h^{n+1}
\right\rangle_h
=
\left\langle
\mathcal R_h^n,
\delta_\tau u_h^{n+1}
\right\rangle_h .
\label{eq_rk2_regularity_inner}
\end{align}
Using
\[
\left\langle
\bar L_h^n u_h^{n+1},
\delta_\tau u_h^{n+1}
\right\rangle_h
=
\frac{1}{2\tau}
\left(
\|u_h^{n+1}\|_{\bar L_h^n}^2
-
\|u_h^n\|_{\bar L_h^n}^2
\right)
+
\frac{\tau}{2}
\|\delta_\tau u_h^{n+1}\|_{\bar L_h^n}^2,
\]
we obtain
\begin{align}
\frac{1}{2\tau}
\left(
\|u_h^{n+1}\|_{\bar L_h^n}^2
-
\|u_h^n\|_{\bar L_h^n}^2
\right)
+
\left\langle
\psi_1(\tau\bar L_h^n)\delta_\tau u_h^{n+1},
\delta_\tau u_h^{n+1}
\right\rangle_h
=
\left\langle
\mathcal R_h^n,
\delta_\tau u_h^{n+1}
\right\rangle_h .
\label{eq_rk2_regularity_energy}
\end{align}
Then the first term in \eqref{eq_rk2_regularity_energy} can be written as
\begin{align}
\frac{1}{2\tau}
\left(
\|u_h^{n+1}\|_{\bar L_h^n}^2
-
\|u_h^n\|_{\bar L_h^n}^2
\right)
={}&
\frac{\varepsilon^2}{2\tau}
\left(
\|\nabla_h u_h^{n+1}\|_h^2
-
\|\nabla_h u_h^n\|_h^2
\right)
+
\frac{\kappa\bar g_h^n}{2\tau}
\left(
\|u_h^{n+1}\|_h^2
-
\|u_h^n\|_h^2
\right)
\no\\
= &
\frac{\varepsilon^2}{2\tau}
\left(
\|\nabla_h u_h^{n+1}\|_h^2
-
\|\nabla_h u_h^n\|_h^2
\right)
+
\frac{\kappa\bar g_h^n}{2}
\left\langle
u_h^{n+1}+u_h^n,
\delta_\tau u_h^{n+1}
\right\rangle_h
.
\label{eq_rk2_L_split}
\end{align}
Substituting \eqref{eq_rk2_L_split} into \eqref{eq_rk2_regularity_energy} gives
\begin{align}
&
\frac{\varepsilon^2}{2\tau}
\left(
\|\nabla_h u_h^{n+1}\|_h^2
-
\|\nabla_h u_h^n\|_h^2
\right)
+
\left\langle
\psi_1(\tau\bar L_h^n)\delta_\tau u_h^{n+1},
\delta_\tau u_h^{n+1}
\right\rangle_h
\nonumber\\
&\qquad=
\left\langle
\mathcal R_h^n,
\delta_\tau u_h^{n+1}
\right\rangle_h
-
\frac{\kappa\bar g_h^n}{2}
\left\langle
u_h^{n+1}+u_h^n,
\delta_\tau u_h^{n+1}
\right\rangle_h .
\label{eq_rk2_regularity_rearranged}
\end{align}
 Young's inequality therefore yields
\begin{align}
\left|
\left\langle
\mathcal R_h^n,
\delta_\tau u_h^{n+1}
\right\rangle_h
\right|
&\le
\frac14
\|\delta_\tau u_h^{n+1}\|_h^2+C,
\nonumber\\
\frac{\kappa\bar g_h^n}{2}
\left|
\left\langle
u_h^{n+1}+u_h^n,
\delta_\tau u_h^{n+1}
\right\rangle_h
\right|
&\le
\frac14
\|\delta_\tau u_h^{n+1}\|_h^2+C.
\label{eq_rk2_regularity_young}
\end{align}
Substituting these estimates into
\eqref{eq_rk2_regularity_rearranged} and using Lemma~\ref{lem_phi_operator_bounds}, we obtain
\begin{equation}
\frac{\varepsilon^2}{2\tau}
\left(
\|\nabla_h u_h^{n+1}\|_h^2
-
\|\nabla_h u_h^n\|_h^2
\right)
+
\frac12
\|\delta_\tau u_h^{n+1}\|_h^2
\le C.
\label{eq_rk2_regularity_step}
\end{equation}
Multiplying \eqref{eq_rk2_regularity_step} by $\tau$ and summing over
$n=0,\ldots,m-1$, where $1\le m\le N_T$, gives
\[
\frac{\varepsilon^2}{2}
\left(
\|\nabla_h u_h^m\|_h^2
-
\|\nabla_h u_h^0\|_h^2
\right)
+
\frac{\tau}{2}
\sum_{n=0}^{m-1}
\|\delta_\tau u_h^{n+1}\|_h^2
\le CT.
\]
The consistent initialization \eqref{eq_ac_initial_data}, the initial energy
bound in Lemma~\ref{lem:g_upper_ac}, and
\eqref{eq_bulk_energy_bounds} imply
$\|\nabla_h u_h^0\|_h\le C$. The MBP controls the discrete $L^2$ norm.
Thus the preceding estimate proves \eqref{eq_ac-uniform-bound}.

For the predictor increment, applying the operator
$\frac1\tau\left(\psi(\tau  L_h^n)
+\tau  L_h^n\right)$
to \eqref{eq_etd1-u} gives the transformed ETD1 equation
\begin{equation}
\psi(\tau L_h^n)\delta_\tau u_h^{*}
+
L_h^n u_h^{*}
=
N_h^n.
\label{eq_predictor_reformulation}
\end{equation}
Taking the discrete inner product with $\delta_\tau u_h^{*}$ and using
\[
\left\langle
L_h^n u_h^{*},\delta_\tau u_h^{*}
\right\rangle_h
=
\frac{1}{2\tau}
\left(
\|u_h^{*}\|_{L_h^n}^2
-
\|u_h^n\|_{L_h^n}^2
\right)
+
\frac{\tau}{2}
\|\delta_\tau u_h^{*}\|_{L_h^n}^2,
\]
we obtain
\begin{align}
\left\langle
\psi_1(\tau L_h^n)\delta_\tau u_h^{*},
\delta_\tau u_h^{*}
\right\rangle_h
+
\frac{1}{2\tau}
\left(
\|u_h^{*}\|_{L_h^n}^2
-
\|u_h^n\|_{L_h^n}^2
\right)
=
\left\langle
N_h^n,\delta_\tau u_h^{*}
\right\rangle_h.
\label{eq_predictor_increment_energy}
\end{align}
By Lemma~\ref{lem_phi_operator_bounds},
\eqref{eq_ac_coarse_nonlinear_increment_bound}, and Young's inequality, we obtain
\[
\frac12\|\delta_\tau u_h^{*}\|_h^2
+
\frac{1}{2\tau}
\left(
\|u_h^{*}\|_{L_h^n}^2
-
\|u_h^n\|_{L_h^n}^2
\right)
\le C.
\]
 Multiplying the preceding inequality by
$2\tau$  and using \eqref{eq_ac-uniform-bound}, we obtain
\[
\tau\|\delta_\tau u_h^{*}\|_h^2
+
\|u_h^{*}\|_{L_h^n}^2
\le
\|u_h^n\|_{L_h^n}^2+C\tau
\le C_T.
\]
Consequently,
\begin{equation}
\|u_h^{*}-u_h^n\|_h^2
=
\tau^2\|\delta_\tau u_h^{*}\|_h^2
\le C_T\tau.
\label{eq_predictor_increment_bound}
\end{equation}

For the auxiliary predictor, define
$\delta_\tau u_h^{\sharp}:=(u_h^{\sharp}-u_h^n)/\tau$.
Testing  \eqref{eq_auxiliary_predictor_reformulated} with
$\delta_\tau u_h^{\sharp}$ and repeating the preceding argument with
$(L_h^n,N_h^n)$ replaced by $(\bar L_h^n,\bar N_h^n)$ gives
\[
\tau\|\delta_\tau u_h^{\sharp}\|_h^2
+\|u_h^{\sharp}\|_{\bar L_h^n}^2
\le
\|u_h^n\|_{\bar L_h^n}^2+C\tau
\le C_T.
\]
Here $\|\bar N_h^n\|_h\le C$ follows from
\eqref{eq_stage-nonlinearity-etd2-bound}, while
$\|u_h^n\|_{\bar L_h^n}^2\le C_T$ follows from
\eqref{eq_ac-uniform-bound}, the MBP, and the upper bound for $\bar g_h^n$.
Consequently, the auxiliary-predictor increment satisfies
\begin{equation}
\|u_h^{\sharp}-u_h^n\|_h^2
=
\tau^2\|\delta_\tau u_h^{\sharp}\|_h^2
\le C_T\tau.
\label{eq_auxiliary_predictor_increment_bound}
\end{equation}
The difference between the actual and auxiliary predictors is therefore
\begin{equation}
\|u_h^{*}-u_h^{\sharp}\|_h^2
\le
2\|u_h^{*}-u_h^n\|_h^2+2\|u_h^{\sharp}-u_h^n\|_h^2
\le C_T\tau,
\label{eq_predictor_auxiliary_difference_bound_proof}
\end{equation}
which proves \eqref{eq_predictor_auxiliary_difference_bound}.
\end{proof}

Define the nonnegative correction dissipation by
\begin{equation}\label{eq_ac_D_definition}
\mathcal D_s^{n+1}:=
\frac34\left\langle
\bar L_h^n (u_h^{n+1}-u_h^{\sharp}),
u_h^{n+1}-u_h^{\sharp}
\right\rangle_h
+\frac47\kappa\bar g_h^n
\|u_h^{*}-u_h^{\sharp}\|_h^2\ge0.
\end{equation}
Its nonnegativity follows from the positive definiteness of $\bar L_h^n$
and the positivity of $\bar g_h^n$.

\begin{lemma}
\label{lem:ac_uniform_g_bounds}
Under the assumptions of Lemma~\ref{lem:discrete_time_regularity}, there
exists a constant $G_->0$, independent of $h$ and $\tau$, such that
\begin{equation}\label{eq_ac_g_two_sided}
\begin{aligned}
G_-\le g_h^n&\le G^+,
&&0\le n\le N_T,\\
G_-\le g_h^{*},\bar g_h^n&\le G^+,
&&0\le n<N_T.
\end{aligned}
\end{equation}
\end{lemma}

\begin{proof}
Because the following argument sums over several time steps, we restore the
step index in the predictor quantities and write
$(u_h^{k,*},s_h^{k,*})$ and
$g_h^{k,*}:=g_h(u_h^{k,*},s_h^{k,*})$.
Summing the ETD2 auxiliary-variable update \eqref{eq_s-correction} gives
for $0\le n\le N_T-1$
\begin{equation}\label{eq_ac_s_recursive_lower}
 s_h^{n+1}=s_h^0-\frac\tau2\sum_{k=0}^{n}
\left\langle g_h^k f(u_h^k)+g_h^{k,*} f(u_h^{k,*}),\delta_\tau u_h^{k+1}\right\rangle_h
-\sum_{k=0}^{n}\mathcal D_s^{k+1}.
\end{equation}
Theorem~\ref{thm:mbp-gsav-etd2},  and Lemma~\ref{lem:g_upper_ac} imply
$\|g_h^k f(u_h^k)\|_h+\|g_h^{k,*} f(u_h^{k,*})\|_h\le C_F$.  Hence Cauchy--Schwarz and
Lemma~\ref{lem:discrete_time_regularity} give
\begin{align}
\frac\tau2\sum_{k=0}^{n}
\left|\left\langle g_h^k f(u_h^k)+g_h^{k,*} f(u_h^{k,*}),
\delta_\tau u_h^{k+1}\right\rangle_h\right|
&\le C_F\sqrt{T}
\left(\tau\sum_{k=0}^{n}
\|\delta_\tau u_h^{k+1}\|_h^2\right)^{1/2}
\le C_T.
\label{eq_ac_s_cauchy_bound}
\end{align}

 From \eqref{eq_etd2-u} and \eqref{eq_auxiliary_predictor1},
\[
u_h^{n+1}-u_h^{\sharp}
=
\tau\varphi_2(\tau\bar L_h^n)
\bigl(\bar N_h^{*}-\bar N_h^n\bigr).
\]
 Hence, taking the discrete inner product with $\bar L_h^n(u_h^{n+1}-u_h^{\sharp})$ and using
 Lemma~\ref{lem_phi_operator_bounds} and
 \eqref{eq_ac_coarse_nonlinear_increment_bound}, we have
\begin{align}
\left\langle
\bar L_h^n(u_h^{n+1}-u_h^{\sharp}),
u_h^{n+1}-u_h^{\sharp}
\right\rangle_h
&=
\tau
\left\langle
(\tau\bar L_h^n)
\varphi_2^2(\tau\bar L_h^n)
(\bar N_h^{*}-\bar N_h^n),
\bar N_h^{*}-\bar N_h^n
\right\rangle_h
\nonumber\\
&\le
\tau
\|\bar N_h^{*}-\bar N_h^n\|_h^2
\le C\tau.
\label{eq_extra_term_1_bound}
\end{align}
Therefore, by \eqref{eq_extra_term_1_bound},
\eqref{eq_predictor_auxiliary_difference_bound}, and
Lemma~\ref{lem:g_upper_ac}, we have
\ba \label{eq_correction_dissipation_step_bound}
\frac34
\left\langle
\bar L_h^n(u_h^{n+1}-u_h^{\sharp}),
u_h^{n+1}-u_h^{\sharp}
\right\rangle_h+\frac47\kappa\bar g_h^n
\|u_h^{*}-u_h^{\sharp}\|_h^2\le C_T\tau,
\ed
which yields
$\sum_{k=0}^{n}\mathcal D_s^{k+1}\le C_T$.

Substituting \eqref{eq_ac_s_cauchy_bound} and
\eqref{eq_correction_dissipation_step_bound} into
\eqref{eq_ac_s_recursive_lower} yields
\begin{equation}\label{eq_ac_s_nodal_lower_bound}
s_h^n\ge s_h^0-C_T,\qquad 0\le n\le N_T,
\end{equation}
where the case $n=0$ is immediate.
By the predictor update \eqref{eq_etd1-s}, Theorem~\ref{thm:mbp-gsav-etd2}, Lemma~\ref{lem:g_upper_ac},
the predictor bound \eqref{eq_predictor-etd2-mbp-bound}, and
\eqref{eq_ac_s_nodal_lower_bound}, we obtain, for $0\le n<N_T$,
\begin{align}
s_h^{*}
&=
s_h^n-
\left\langle
g_h^n f(u_h^n),u_h^{*}-u_h^n
\right\rangle_h
\nonumber\\
&\ge
s_h^0-C_T-
\left|
\left\langle
g_h^n f(u_h^n),u_h^{*}-u_h^n
\right\rangle_h
\right|
\ge s_h^0-C_T-2C_F\beta|\Omega|^{1/2}.
\label{eq_ac_s_predictor_lower_bound}
\end{align}
By \eqref{eq_ac_initial_data} and \eqref{eq_bulk_energy_bounds},
$s_h^0\ge-C_E$, so both lower bounds are uniform in $h$ and $\tau$.
The predictor bound \eqref{eq_predictor-etd2-mbp-bound} also gives
$|E_{1h}(u_h^{*})|\le C_E$. Then the positivity and
monotonicity of $\sigma$ in \eqref{eq_sigma_conditions} imply that there exists a constant $G_->0$,
independent of $h$ and $\tau$, such that
\[
g_h^n,\; g_h^{*},\; \bar g_h^n
\ge \frac{\sigma(-C_E-C_T-2C_F\beta|\Omega|^{1/2})}{\sigma(C_E)}=:G_->0.
\]
 Combining these estimates with
the upper bounds in Lemma~\ref{lem:g_upper_ac} proves
\eqref{eq_ac_g_two_sided}.
\end{proof}
\section{Convergence analysis for the GSAV--ETD2 scheme}
Throughout this section, $(u_h^n,s_h^n)$ denotes the numerical solution of
\eqref{eq_etd1-u}--\eqref{eq_s-correction}, and $(u(t),s(t))$ denotes the
exact solution of \eqref{eq_ac_continuous_gsav}--\eqref{eq_ac_continuous_gsav1}. The letter $C$ denotes a
generic positive constant independent of $h$ and $\tau$, whose value may
change from line to line. We assume that the exact solution $u$ satisfies
\begin{equation}\label{eq:exact_solution_regularity}
\begin{aligned}
u\in W^{1,\infty}(0,T;H^4(\Omega))
\cap W^{2,\infty}(0,T;H^2(\Omega)).
\end{aligned}
\end{equation}
We define the following errors between the numerical and exact solutions:
\begin{equation}\label{eq_ac_error_notation}
\left(
e_u^n,e_s^n,e_u^{*},e_s^{*}
\right)
:=
\left(
u_h^n-\mathcal I_h u(t_n),
s_h^n-s(t_n),
u_h^{*}-\mathcal I_h u(t_{n+1}),
s_h^{*}-s(t_{n+1})
\right).
\end{equation}
\begin{lemma}[Periodic-grid consistency]\label{lem:ac_periodic_grid_consistency}
For periodic $v,w\in H^2(\Omega)$, the nodal restriction satisfies
\begin{align}
\|\mathcal I_h v\|_h
&\le C\|v\|_{H^2},\qquad
\left|
\langle\mathcal I_h v,\mathcal I_h w\rangle_h-(v,w)
\right|
\le Ch^2\|v\|_{H^2}\|w\|_{H^2}.
\label{eq_ac_sampling_quadrature_consistency}
\end{align}
Under \eqref{eq:exact_solution_regularity}, uniformly for $t\in[0,T]$,
\begin{equation}
\|\chi_h\|_{L^\infty(0,T;\ell_h^2)}
+
\|\partial_t\chi_h\|_{L^\infty(0,T;\ell_h^2)}
+
\sup_{0\le t\le T}
|E_{1h}(\mathcal I_h u(t))-E_1(u(t))|
\le Ch^2,
\label{eq_ac_spatial_consistency}
\end{equation}
where
$\chi_h(t)
:=
\varepsilon^2
\bigl(
\mathcal I_h \Delta u(t)
-\Delta_h \mathcal I_h u(t)
\bigr)$.
\end{lemma}
\begin{proof}
The first two estimates follow from the periodic trapezoidal-rule error and
the embedding $H^2(\Omega)\hookrightarrow C(\overline\Omega)$ in two
dimensions. The estimate for $\chi_h$ and its time derivative follows from
the second-order consistency of the periodic central-difference Laplacian and
\eqref{eq:exact_solution_regularity}. Finally, the energy estimate follows
from the trapezoidal-rule
error, $F\in C(\mathbb R)$, and the $H^2$-algebra property in two
dimensions.
\end{proof}

In the remainder of this section, $\mathcal I_h$ is suppressed whenever an
exact solution appears inside a discrete operator, norm, or inner product.


\subsection{Local Lipschitz continuity and prediction-stage error estimate}
In this subsection, we derive a local Lipschitz continuity property for the
nonlinear term and a prediction-stage error estimate.
We first introduce the continuous counterpart of the nonlinear term for the
Allen--Cahn equation.  For every
$t\in[t_n,t_{n+1}]$, define
\[
N_{\rm ex}(t):=g(t)f(u(t))+\kappa\bar g_h^n u(t),
\]
where $\bar g_h^n$ is the frozen GSAV factor defined immediately before
\eqref{eq_etd2-u} and $g(t):=g(u(t),s(t))\equiv1$ in the continuous case.
The dependence of $N_{\rm ex}$ on the step index $n$ through the frozen
factor $\bar g_h^n$ is suppressed to keep the notation readable.

Substituting the exact solution into the GSAV--ETD2 formulation and using
linear interpolation for the nonlinear term, we obtain
\begin{align}
u(t_{n+1})
={}&
e^{-\tau \bar L_h^n}u(t_n)
+
\int_0^\tau
e^{-(\tau-\theta)\bar L_h^n}
\left[
\left(1-\frac{\theta}{\tau}\right)N_{\rm ex}(t_n)
+
\frac{\theta}{\tau}N_{\rm ex}(t_{n+1})
+
R_u^n(\theta)
\right]\,\mathrm d\theta ,
\label{eq_ac_exact_etd2}
\\
\delta_\tau s(t_{n+1})
={}&
-
\left\langle
\frac12\left(
g(t_n)f(u(t_n))+g(t_{n+1})f(u(t_{n+1}))
\right),
\delta_\tau u(t_{n+1})
\right\rangle_h
+
R_s^n ,
\label{eq_ac_exact_s}
\end{align}
where
$\delta_\tau u(t_{n+1})
=
(u(t_{n+1})-u(t_n))/\tau$ and  the truncation errors are defined by
\begin{align}
R_u^n(\theta)
&:=N_{\rm ex}(t_n+\theta)
-\left(1-\frac\theta\tau\right)N_{\rm ex}(t_n)
-\frac\theta\tau N_{\rm ex}(t_{n+1})
+\chi_h(t_n+\theta),
\label{eq_ac_Ru_definition}\\
R_s^n
&:=\delta_\tau s(t_{n+1})
+\left\langle
\frac{g(t_n)f(u(t_n))+g(t_{n+1})f(u(t_{n+1}))}{2},
\delta_\tau u(t_{n+1})\right\rangle_h .
\label{eq_ac_Rs_definition}
\end{align}

Applying the transformation
$\frac1\tau\left(\psi(\tau \bar L_h^n)
+\tau \bar L_h^n\right)$
to \eqref{eq_ac_exact_etd2} gives
\begin{align}
\psi(\tau \bar L_h^n)\delta_\tau u(t_{n+1})
+
\bar L_h^n u(t_{n+1})
={}&
N_{\rm ex}(t_n)
+
\frac{\varphi_2(\tau \bar L_h^n)}
{\varphi_1(\tau \bar L_h^n)}
\left(
N_{\rm ex}(t_{n+1})-N_{\rm ex}(t_n)
\right)
+
\widetilde R_u^n ,
\label{eq_ac_exact_transformed}
\end{align}
where
\ba\label{eq:transformed_residual_kernel}
\widetilde R_u^n
:=
\int_0^\tau
\mathcal K_n(\theta)
R_u^n(\theta)\,\mathrm d\theta,\quad\mathcal K_n(\theta)
:=
\tau^{-1}\bigl(\psi(\tau \bar L_h^n)
+\tau \bar L_h^n\bigr)
e^{-(\tau-\theta)\bar L_h^n}.
\ed
For an eigenvalue $\lambda\ge0$ of $\bar L_h^n$ and
$q=(\tau-r)/\tau\in[0,1]$, the scalar multiplier of
$\int_r^\tau\mathcal K_n(\theta)\,\mathrm d\theta$ is
$(1-e^{-q\tau\lambda})/(1-e^{-\tau\lambda})\in[0,1]$, with its value at
$\lambda=0$ understood by continuity. Hence spectral calculus gives
\[
\int_0^\tau \mathcal K_n(\theta)\,\mathrm d\theta=I,
\qquad
\vertiii{
\int_r^\tau \mathcal K_n(\theta)\,\mathrm d\theta
}_{h}\le1,
\qquad 0\le r\le\tau.
\]
Applying Taylor's formula for linear interpolation to \eqref{eq_ac_Ru_definition}, for any $0\le\theta\le\tau$, we obtain
\begin{align*}
R_u^n(\theta)
={}-\frac{\tau-\theta}{\tau}
\int_0^\theta r\,(N_{\rm ex})_{tt}(t_n+r)\,\mathrm dr
-\frac\theta\tau\int_\theta^\tau(\tau-r)(N_{\rm ex})_{tt}(t_n+r)\,\mathrm dr
+\chi_h(t_n+\theta).
\end{align*}
The assumptions in \eqref{eq:exact_solution_regularity},
Lemma~\ref{lem:ac_uniform_g_bounds}, and the spatial consistency estimate
\eqref{eq_ac_spatial_consistency} imply
\ba\label{eq:interpolation_residual_bounds}
\|R_u^n(0)\|_h\le Ch^2,
\qquad
\|\partial_\theta R_u^n(\theta)\|_h\le C(\tau+h^2).
\ed
Consequently, for $0<\tau\le1$, taking the $\ell_h^2$ norm of
\eqref{eq:transformed_residual_kernel} and using
\eqref{eq:interpolation_residual_bounds} and Fubini's theorem gives
\[
\|\widetilde R_u^n\|_h=\left\|R_u^n(0)+\int_0^\tau\left(\int_r^\tau\mathcal K_n(\theta)\,\mathrm d\theta\right)
\partial_rR_u^n(r)\,\mathrm dr\right\|_h\le C(\tau^2+h^2).
\]
Subtracting \eqref{eq_ac_exact_transformed} from the transformed numerical
equation \eqref{eq_ac_rk2_regularity_equation} gives the error equation for
$u$:
\begin{align}
&\psi(\tau \bar L_h^n)\delta_\tau e_u^{n+1}
+
\bar L_h^n e_u^{n+1}\no\\
{}&\qquad=
\left(I-
\frac{\varphi_2(\tau \bar L_h^n)}
{\varphi_1(\tau \bar L_h^n)}\right)
\left(\bar N_h^n-N_{\rm ex}(t_n)\right)
+
\frac{\varphi_2(\tau \bar L_h^n)}
{\varphi_1(\tau \bar L_h^n)}
\left(\bar N_h^{*}-N_{\rm ex}(t_{n+1})\right)
-
\widetilde R_u^n .
\label{eq_ac_error_u}
\end{align}

Subtracting \eqref{eq_ac_exact_s} from \eqref{eq_s-correction} gives the
auxiliary-variable error equation:
\begin{align}
\delta_\tau e_s^{n+1}
={}&
-
\left\langle
\frac12(g_h^n f(u_h^n)+g_h^{*} f(u_h^{*})),
\delta_\tau e_u^{n+1}
\right\rangle_h-
R_s^n
-
\frac1\tau\mathcal D_s^{n+1}
\nonumber\\
&-
\left\langle
\frac12
\left[
g_h^n f(u_h^n)-g(t_n)f(u(t_n))
+g_h^{*} f(u_h^{*})-g(t_{n+1})f(u(t_{n+1}))
\right],
\delta_\tau u(t_{n+1})
\right\rangle_h
,
\label{eq_ac_error_s}
\end{align}
where
\begin{equation}\label{eq_ac_Rs_covariance}
\begin{aligned}
R_s^n={}&
\left\langle
\frac{g(t_n)f(u(t_n))+g(t_{n+1})f(u(t_{n+1}))}{2}
-\frac1\tau\int_{t_n}^{t_{n+1}}g(t)f(u(t))\,\mathrm dt,
\delta_\tau u(t_{n+1})
\right\rangle_h
\\
&-\frac1\tau\int_{t_n}^{t_{n+1}}
\left\langle
g(t)f(u(t))
-\frac1\tau\int_{t_n}^{t_{n+1}}g(r)f(u(r))\,\mathrm dr,
u_t(t)-\delta_\tau u(t_{n+1})
\right\rangle_h\,\mathrm dt
\\
&+\frac1\tau\int_{t_n}^{t_{n+1}}
\left(
\langle g(t)f(u(t)),u_t(t)\rangle_h
-(g(t)f(u(t)),u_t(t))
\right)\,\mathrm dt.
\end{aligned}
\end{equation}
By \eqref{eq:exact_solution_regularity} and the sampling consistency estimate
\eqref{eq_ac_sampling_quadrature_consistency}, the endpoint trapezoidal error satisfies
\[
\left\|
\frac{g(t_n)f(u(t_n))+g(t_{n+1})f(u(t_{n+1}))}{2}
-\frac1\tau\int_{t_n}^{t_{n+1}}g(t)f(u(t))\,\mathrm dt
\right\|_h
\le C\tau^2.
\]
Moreover, uniformly for $t\in[t_n,t_{n+1}]$, the regularity condition
\eqref{eq:exact_solution_regularity} and the sampling consistency estimate
\eqref{eq_ac_sampling_quadrature_consistency}, we have
\[
\left\|
g(t)f(u(t))
-\frac1\tau\int_{t_n}^{t_{n+1}}g(r)f(u(r))\,\mathrm dr
\right\|_h
\le C\tau,
\qquad
\left\|
u_t(t)-\delta_\tau u(t_{n+1})
\right\|_h
\le C\tau.
\]
Therefore, using also the quadrature consistency estimate
\eqref{eq_ac_sampling_quadrature_consistency}, we obtain
\begin{align}
|R_s^n|
&\le
C\tau^2
+
\frac{C}{\tau}
\int_{t_n}^{t_{n+1}}
\tau^2\,\mathrm dt
+
\frac{C}{\tau}
\int_{t_n}^{t_{n+1}}
h^2\,\mathrm dt
\nonumber\\
&\le
C(\tau^2+h^2).
\end{align}  Thus
\begin{equation}\label{eq_ac_residual_bound}
\|\widetilde R_u^n\|_h+|R_s^n|
\le
C(\tau^2+h^2).
\end{equation}
Then by Theorem~\ref{thm:mbp-gsav-etd2}, Lemma~\ref{lem:ac_uniform_g_bounds}, and
\eqref{eq:exact_solution_regularity}, we have the uniform bounds
\begin{equation}\label{eq_ac_coarse_numerical_bounds}
\|g_h^n f(u_h^n)\|_h
+\|g_h^{*} f(u_h^{*})\|_h
+\|\bar N_h^n\|_h
+\|\bar N_h^{*}\|_h
+\|\mathcal I_h N_{\rm ex}(t_n)\|_h
+\|\mathcal I_h N_{\rm ex}(t_{n+1})\|_h
\le C.
\end{equation}

To close these coupled error equations, we first record uniform Lipschitz
bounds for the nonlinear quantities, then propagate the nodal error through
the predictor, and finally estimate the additional correction defect.

\begin{lemma}[Local Lipschitz estimates]\label{lem:ac_lipschitz}
Suppose that the exact solution satisfies
\eqref{eq:exact_solution_regularity} and the consistency
estimates in Lemma~\ref{lem:ac_periodic_grid_consistency}, and assume the
hypotheses of Lemma~\ref{lem:ac_uniform_g_bounds}. Then there exists a
constant $C>0$, independent
of $\tau$ and $h$, such
that
\begin{align}
|g_h^n-g(t_n)|
&\le
C\left(
|e_s^n|+\|e_u^n\|_h+h^2
\right),
\label{eq_ac_g_lip_n}
\\
|g_h^{*}-g(t_{n+1})|
&\le
C\left(
|e_s^{*}|+\|e_u^{*}\|_h+h^2
\right),
\label{eq_ac_g_lip_star}
\\
\|g_h^n f(u_h^n)-g(t_n)f(u(t_n))\|_h
&\le
C\left(
|e_s^n|+\|e_u^n\|_h+h^2
\right),
\label{eq_ac_F_lip_n}
\\
\|g_h^{*} f(u_h^{*})-g(t_{n+1})f(u(t_{n+1}))\|_h
&\le
C\left(
|e_s^{*}|+\|e_u^{*}\|_h+h^2
\right),
\label{eq_ac_F_lip_star}
\\
\|\bar N_h^n-N_{\rm ex}(t_n)\|_h
&\le
C\left(
|e_s^n|+\|e_u^n\|_h+h^2
\right),
\label{eq_ac_N_lip_n}
\\
\|\bar N_h^{*}-N_{\rm ex}(t_{n+1})\|_h
&\le
C\left(
|e_s^{*}|+\|e_u^{*}\|_h+h^2
\right).
\label{eq_ac_N_lip_star}
\end{align}
\end{lemma}
\begin{proof}
By Theorem~\ref{thm:mbp-gsav-etd2}, the predictor bound
\eqref{eq_predictor-etd2-mbp-bound}, the mean value theorem, and
Lemma~\ref{lem:ac_periodic_grid_consistency},
\[
\left|
E_{1h}(u_h^n)-E_1(u(t_n))
\right|
\le
C\left(
\|e_u^n\|_h+h^2
\right),
\]
and the same estimate holds for the predictor at $t_{n+1}$.
The bulk-energy bound \eqref{eq_bulk_energy_bounds}, the scalar upper bounds
\eqref{eq_ac_s_upper_bounds}, and the lower bounds
\eqref{eq_ac_s_nodal_lower_bound} and \eqref{eq_ac_s_predictor_lower_bound},
together with \eqref{eq:exact_solution_regularity} and
$s(t)=E_1(u(t))$, ensure that all arguments
of $\sigma$ remain in a fixed compact interval. Since
$\sigma\in C^1(\mathbb R)$ is strictly positive, the mapping
\[
(s,E)\mapsto \frac{\sigma(s)}{\sigma(E)}
\]
is Lipschitz on this interval. Hence, by the mean value theorem,
\eqref{eq_ac_g_lip_n} and \eqref{eq_ac_g_lip_star} follow.

Moreover, the MBP and $f\in C^1(\mathbb R)$ imply
\[
\|f(u_h^n)-f(u(t_n))\|_h
\le C\|e_u^n\|_h,
\]
and similarly for the predictor. Together with the uniform bounds on
$g_h^n$ and $g_h^{*}$, this gives
\eqref{eq_ac_F_lip_n} and \eqref{eq_ac_F_lip_star}.

Finally, using
\[
\bar N_h^n-N_{\rm ex}(t_n)
=
g_h^n f(u_h^n)-g(t_n)f(u(t_n))
+
\kappa\bar g_h^n e_u^n
\]
and the uniform bound \eqref{eq_ac_g_two_sided}, we obtain
\eqref{eq_ac_N_lip_n}. The estimate
\eqref{eq_ac_N_lip_star} follows analogously.
\end{proof}
\begin{lemma}[Prediction-stage error estimate]\label{lem:ac_predictor_error}
Suppose that $u(t)$ satisfies \eqref{eq:exact_solution_regularity} and the
consistency estimates in
Lemma~\ref{lem:ac_periodic_grid_consistency}, and assume the hypotheses of
Lemma~\ref{lem:ac_lipschitz}.
Then
there exists a constant $C>0$, independent of $\tau$ and $h$, such that
\begin{equation}\label{eq_ac_star_error_bound}
\|e_u^{*}\|_h+|e_s^{*}|
\le
C\left(
\|e_u^n\|_h+|e_s^n|+\tau(\tau+h^2)
\right).
\end{equation}
\end{lemma}

\begin{proof}
The variation-of-constants formula for the exact solution, with
$L_h^n$ frozen on $[t_n,t_{n+1}]$, gives
\begin{align}
e_u^{*}={}&e^{-\tau L_h^n}e_u^n
+\tau\varphi_1(\tau L_h^n)
\bigl(g_h^n f(u_h^n)-g(t_n)f(u(t_n))+\kappa g_h^n e_u^n\bigr)
-R_{p,u}^n,
\label{eq_ac_star_error_voc}
\end{align}
where, with $\chi_h$ from \eqref{eq_ac_spatial_consistency},
\begin{equation*}
R_{p,u}^n
=\int_0^\tau e^{-(\tau-r)L_h^n}
\left[
g(t_n+r)f(u(t_n+r))-g(t_n)f(u(t_n))
+\kappa g_h^n\bigl(u(t_n+r)-u(t_n)\bigr)
+\chi_h(t_n+r)
\right]\mathrm dr .
\end{equation*}
Taylor's formula, \eqref{eq_ac_sampling_quadrature_consistency}, and
\eqref{eq_ac_spatial_consistency} imply
\begin{equation}\label{eq_ac_predictor_residual}
\|R_{p,u}^n\|_h\le C\tau(\tau+h^2).
\end{equation}
Taking the $h$-norm of \eqref{eq_ac_star_error_voc} and using
Lemmas~\ref{lem:ac_lipschitz} and~\ref{lem_phi_operator_bounds}, together with the
uniform bound \eqref{eq_ac_g_two_sided}, we obtain
\begin{equation}\label{eq_ac_star_L2_bound}
\|e_u^{*}\|_h
\le C\bigl(\|e_u^n\|_h+\tau|e_s^n|
+\tau(\tau+h^2)\bigr).
\end{equation}

It remains to estimate $e_s^{*}$. From the predictor update,
\[
s_h^{*}
=
s_h^n
-
\left\langle
g_h^n f(u_h^n),u_h^{*}-u_h^n
\right\rangle_h.
\]
Integrating the exact auxiliary equation over $[t_n,t_{n+1}]$ and inserting
the left-endpoint approximation gives
\[
s(t_{n+1})
=
s(t_n)
-
\left\langle
g(t_n)f(u(t_n)),u(t_{n+1})-u(t_n)
\right\rangle_h
+
R_{p,s}^n,
\]
where
\begin{align}
R_{p,s}^n
:={}&
\left\langle
g(t_n)f(u(t_n)),u(t_{n+1})-u(t_n)
\right\rangle_h
-
\int_{t_n}^{t_{n+1}}(g(t)f(u(t)),u_t(t))\,\mathrm dt
\nonumber\\
={}&
\int_{t_n}^{t_{n+1}}
\left\langle
g(t_n)f(u(t_n))-g(t)f(u(t)),u_t(t)
\right\rangle_h\,\mathrm dt
\nonumber\\
&+
\int_{t_n}^{t_{n+1}}
\left(
\langle g(t)f(u(t)),u_t(t)\rangle_h
-(g(t)f(u(t)),u_t(t))
\right)\,\mathrm dt.
\label{eq_ac_predictor_scalar_residual}
\end{align}
The first integral is $O(\tau^2)$ by
\eqref{eq:exact_solution_regularity} and the sampling
estimate in \eqref{eq_ac_sampling_quadrature_consistency}; the second is
$O(\tau h^2)$ by the quadrature estimate in the same equation. Therefore,
\begin{equation}\label{eq_ac_predictor_scalar_residual_bound}
|R_{p,s}^n|\le C\tau(\tau+h^2).
\end{equation}
Subtracting the two identities, we obtain
\begin{align}
e_s^{*}
={}&
e_s^n
-
\left\langle
g_h^n f(u_h^n),e_u^{*}-e_u^n
\right\rangle_h
\nonumber\\
&-
\left\langle
g_h^n f(u_h^n)-g(t_n)f(u(t_n)),
u(t_{n+1})-u(t_n)
\right\rangle_h
-
R_{p,s}^n .
\label{eq_ac_star_es_identity}
\end{align}
By Theorem~\ref{thm:mbp-gsav-etd2}, Lemma~\ref{lem:g_upper_ac}, and
\eqref{eq_ac_star_L2_bound}, we have
\ba \label{eq:predictor_scalar_increment_bound}
\left|\left\langle
g_h^n f(u_h^n),e_u^{*}-e_u^n
\right\rangle_h\right|\le
C\bigl(\|e_u^n\|_h+\tau|e_s^n|+\tau(\tau+h^2)\bigr).
\ed
Substituting \eqref{eq_ac_g_lip_n},
\eqref{eq_ac_F_lip_n},
\eqref{eq_ac_predictor_scalar_residual_bound}, and \eqref{eq_ac_star_L2_bound} into \eqref{eq_ac_star_es_identity} gives
\[
|e_s^{*}|
\le
C\left(
|e_s^n|
+
\|e_u^n\|_h
+
\tau(\tau+h^2)
\right).
\]
Combining this estimate with \eqref{eq_ac_star_L2_bound} proves
\eqref{eq_ac_star_error_bound}.
\end{proof}
\subsection{Error estimate}
\begin{lemma}[Estimate of the correction defect]\label{lem:ac_correction_defect}
Assume the hypotheses of Lemmas~\ref{lem:discrete_time_regularity} and
\ref{lem:ac_predictor_error}. Then, for
$0<\tau,h\le1$,
\begin{equation}\label{eq_ac_D_bound}
\frac{1}{\tau}\mathcal D_s^{n+1}
\le
C\left(
\|e_u^n\|_h+|e_s^n|+\tau^2+h^2
\right),
\end{equation}
where $\mathcal D_s^{n+1}$ is defined by \eqref{eq_ac_D_definition} and
$C>0$ is independent of $\tau$ and $h$.
\end{lemma}

\begin{proof}
Subtracting the auxiliary predictor $u_h^{\sharp}$ from
\eqref{eq_etd2-u} gives the exact identity
\begin{equation}\label{eq_ac_correction_increment_identity}
u_h^{n+1}-u_h^{\sharp}
=\tau\varphi_2(\tau \bar L_h^n)
\left(\bar N_h^{*}-\bar N_h^n\right).
\end{equation}
Set $w_h:=u_h^{n+1}-u_h^{\sharp}$, consistently with
\eqref{eq_virtual_time_increments}. By the weighted operator bound in
Lemma~\ref{lem_phi_operator_bounds}, spectral calculus yields
\begin{equation}\label{eq_ac_correction_increment_energy_bound}
\frac1\tau\left\langle
\bar L_h^n w_h,
w_h\right\rangle_h
\le C\|\bar N_h^{*}-\bar N_h^n\|_h^2.
\end{equation}
Insert the exact nonlinearities at $t_n$ and $t_{n+1}$:
\[
\bar N_h^{*}-\bar N_h^n
=
\bigl(\bar N_h^{*}-N_{\rm ex}(t_{n+1})\bigr)
+
\bigl(N_{\rm ex}(t_{n+1})-N_{\rm ex}(t_n)\bigr)
+
\bigl(N_{\rm ex}(t_n)-\bar N_h^n\bigr).
\]
Lemmas~\ref{lem:ac_lipschitz} and~\ref{lem:ac_predictor_error}, together
with \eqref{eq:exact_solution_regularity},
Lemma~\ref{lem:ac_periodic_grid_consistency}, and the
uniform bounds \eqref{eq_ac_coarse_numerical_bounds}, yield the following
estimate. In passing from the first line below to the second, we use the
uniform bounds on the exact and numerical nonlinearities, so that the square
of each bounded difference is controlled by a constant times that difference.
\begin{align}
\frac1\tau
\left\langle
\bar L_h^n w_h,w_h
\right\rangle_h
&\le
C\left(
\left\|
\bar N_h^{*}-N_{\rm ex}(t_{n+1})
\right\|_h^2
+
\left(
\int_{t_n}^{t_{n+1}}
\left\|
\partial_t N_{\rm ex}(t)
\right\|_h\,\mathrm dt
\right)^2
+
\left\|
N_{\rm ex}(t_n)-\bar N_h^n
\right\|_h^2
\right)
\nonumber\\
&\le
C\left(
\left\|
\bar N_h^{*}-N_{\rm ex}(t_{n+1})
\right\|_h
+
\tau^2
+
\left\|
N_{\rm ex}(t_n)-\bar N_h^n
\right\|_h
\right)
\nonumber\\
&\le
C\left(
\|e_u^{*}\|_{h}
+
|e_s^{*}|
+
\|e_u^n\|_{h}
+
|e_s^n|
+
\tau^2
+
h^2
\right)
\nonumber\\
&\le
C\left(
\|e_u^n\|_{h}
+
|e_s^n|
+
\tau^2
+
h^2
\right).
\label{eq_ac_correction_increment_energy_bound1}
\end{align}

It remains to estimate the mismatch between the two predictors, which
arises from the different frozen stabilization factors. By the definitions
of $u_h^{*}$ and $u_h^{\sharp}$, we have
\begin{align}
u_h^{*}-u_h^n
&=
\left(e^{-\tau L_h^n}-I\right)u_h^n
+\tau\varphi_1(\tau L_h^n)N_h^n
\nonumber\\
&=
-\int_0^\tau e^{-rL_h^n}L_h^n u_h^n\,\mathrm dr
+\int_0^\tau e^{-rL_h^n}N_h^n\,\mathrm dr
\nonumber\\
&=
\int_0^\tau e^{-rL_h^n}
\left(N_h^n-L_h^n u_h^n\right)\,\mathrm dr
\nonumber\\
&=
\int_0^\tau e^{-rL_h^n}
\bigl(
g_h^n f(u_h^n)+\varepsilon^2\Delta_h u_h^n
\bigr)\,\mathrm dr,\label{eq:predictor_integral_representation}
\end{align}
and
\begin{align}
u_h^{\sharp}-u_h^n
=
\int_0^\tau
e^{-r\bar L_h^n}
\bigl(
g_h^n f(u_h^n)+\varepsilon^2\Delta_h u_h^n
\bigr)\,\mathrm dr.\label{eq:auxiliary_predictor_integral_representation}
\end{align}
Subtracting \eqref{eq:auxiliary_predictor_integral_representation} from
\eqref{eq:predictor_integral_representation} gives
\begin{equation}
u_h^{*}-u_h^{\sharp}
=
\int_0^\tau
e^{-rL_h^n}
\left(
I-e^{-r\kappa(\bar g_h^n-g_h^n)}
\right)
\bigl(
g_h^n f(u_h^n)+\varepsilon^2\Delta_h u_h^n
\bigr)\,\mathrm dr.
\label{eq_ac_d_integral}
\end{equation}
Taking the $\ell_h^2$-norm on both sides of \eqref{eq_ac_d_integral} gives
\ba\label{eq_ac_d_integral_norm}
\|u_h^{*}-u_h^{\sharp}\|_h
&\le
\int_0^\tau
\left\|e^{-rL_h^n}
\left(I-e^{-r\kappa(\bar g_h^n-g_h^n)}\right)
\bigl(g_h^n f(u_h^n)+\varepsilon^2\Delta_h u_h^n\bigr)
\right\|_h\,\mathrm dr,\\
&\le
\int_0^\tau \left|1-e^{-r\kappa(\bar g_h^n-g_h^n)}
\right|
\left\|e^{-rL_h^n}\bigl(g_h^n f(u_h^n)+\varepsilon^2\Delta_h u_h^n\bigr)\right\|_h\,\mathrm dr\\
&\le \int_0^\tau \kappa r|\bar g_h^n-g_h^n|
\left\|e^{-rL_h^n}\bigl(g_h^n f(u_h^n)+\varepsilon^2\Delta_h u_h^n\bigr)\right\|_h\,\mathrm dr.
\ed
By Lemma~\ref{lem:discrete_semigroup_estimates},
\eqref{eq_ac-uniform-bound}, and
\eqref{eq_ac_coarse_numerical_bounds}, with
$a=r\varepsilon^2$ and $b=r\kappa g_h^n$, we have
\begin{align}
\left\|
e^{-rL_h^n}
\bigl(
g_h^n f(u_h^n)+\varepsilon^2\Delta_h u_h^n
\bigr)
\right\|_h&\le
\left\|
e^{-rL_h^n}g_h^n f(u_h^n)
\right\|_h
+
\left\|
e^{-rL_h^n}\varepsilon^2\Delta_h u_h^n
\right\|_h
\nonumber\\
&\le
\|g_h^n f(u_h^n)\|_h
+
Cr^{-1/2}\|\nabla_h u_h^n\|_h
\nonumber\\
&\le
C\bigl(1+r^{-1/2}\bigr).
\label{eq_ac_integrand_smoothing}
\end{align}
Substituting \eqref{eq_ac_integrand_smoothing} into
\eqref{eq_ac_d_integral_norm} and using the two-sided factor bound
\eqref{eq_ac_g_two_sided} gives
\ba\label{eq_ac_predictor_mismatch_energy_bound}
\|u_h^{*}-u_h^{\sharp}\|_h
\le
\kappa|\bar g_h^n-g_h^n|\int_0^\tau r(1+r^{-1/2})\,\mathrm dr
\le
C\tau^\frac{3}{2},
\ed
which implies $
\frac{1}{\tau}\|u_h^{*}-u_h^{\sharp}\|_h^2
\le
C\tau^2$.

Combining \eqref{eq_ac_correction_increment_energy_bound1} and
\eqref{eq_ac_predictor_mismatch_energy_bound} with the definition
\eqref{eq_ac_D_definition}, we obtain
\[
0
\le
\frac{1}{\tau}\mathcal D_s^{n+1}
\le
C\left(
\|e_u^n\|_h+|e_s^n|+\tau^2+h^2
\right),
\]
which yields
\eqref{eq_ac_D_bound}.
\end{proof}

The preceding residual, predictor, and correction estimates supply all
right-hand-side terms required by the fixed-energy argument below.
\begin{theorem}[Fully discrete error estimate]\label{thm_global_error}
Assume the regularity condition \eqref{eq:exact_solution_regularity} and the
consistency estimates
in Lemma~\ref{lem:ac_periodic_grid_consistency}. Let $(u_h^n,s_h^n)$ be the
numerical solution generated by the stabilized GSAV--ETD2 scheme
\eqref{eq_etd1-u}--\eqref{eq_s-correction}, with the consistent
initialization \eqref{eq_ac_initial_data}. Assume the hypotheses of
Theorem~\ref{thm:mbp-gsav-etd2} and
$\|u(t)\|_{L^\infty(\Omega)}\le\beta$ for $0\le t\le T$. Then there exist
$\tau_0,h_0\in(0,1]$ and a positive constant $C$, independent of $\tau$ and $h$,
such that, whenever $0<\tau\le\tau_0$ and $0<h\le h_0$,
\begin{equation}\label{eq_final_total_error}
\|e_u^n\|_{H_h^1}+|e_s^n|
\le
C(\tau^2+h^2),
\qquad
0\le n\le N_T .
\end{equation}
\end{theorem}

\begin{proof}
The initialization \eqref{eq_ac_initial_data},
Lemma~\ref{lem:ac_periodic_grid_consistency}, and
\eqref{eq_bulk_energy_bounds} first give
$\mathcal E_h(u_h^0,s_h^0)\le C_0$ uniformly in $h$. Therefore,
Theorem~\ref{thm:mbp-gsav-etd2} and
Lemmas~\ref{lem:g_upper_ac}, \ref{lem:discrete_time_regularity}, and
\ref{lem:ac_uniform_g_bounds} provide, in this order, all uniform numerical
bounds used below.
Taking the inner product of \eqref{eq_ac_error_u} with $\delta_\tau e_u^{n+1}$ gives
\begin{align}
&\left\langle
\psi(\tau \bar L_h^n)\delta_\tau e_u^{n+1},\delta_\tau e_u^{n+1}
\right\rangle_h
+
\left\langle
\bar L_h^n e_u^{n+1},\delta_\tau e_u^{n+1}
\right\rangle_h\no\\\qquad
={}&
\left\langle
\left(I-
\frac{\varphi_2(\tau \bar L_h^n)}
{\varphi_1(\tau \bar L_h^n)}\right)
\left(\bar N_h^n-N_{\rm ex}(t_n)\right)
+
\frac{\varphi_2(\tau \bar L_h^n)}
{\varphi_1(\tau \bar L_h^n)}
\left(\bar N_h^{*}-N_{\rm ex}(t_{n+1})\right)
-
\widetilde R_u^n  ,\delta_\tau e_u^{n+1}
\right\rangle_h.
\label{eq_ac_global_A_identity}
\end{align}
For the left-hand side of \eqref{eq_ac_global_A_identity}, we have
\begin{align}
\mathrm{LHS}
&=
\left\langle
\psi(\tau \bar L_h^n)\delta_\tau e_u^{n+1},
\delta_\tau e_u^{n+1}
\right\rangle_h
+
\left\langle
\bar L_h^n e_u^{n+1},
\delta_\tau e_u^{n+1}
\right\rangle_h
\nonumber\\
&=
\left\langle
\left(
\psi(\tau \bar L_h^n)
+\frac{\tau}{2}\bar L_h^n
\right)
\delta_\tau e_u^{n+1},
\delta_\tau e_u^{n+1}
\right\rangle_h
\nonumber\\
&\quad
+
\frac{1}{2\tau}
\left(
\|e_u^{n+1}\|_{\bar L_h^n}^2
-
\|e_u^n\|_{\bar L_h^n}^2
\right).
\end{align}

Since
\[
\bar L_h^n
=
-\varepsilon^2\Delta_h+\kappa \bar g_h^n I
=
-\varepsilon^2\Delta_h+\kappa G_- I
+\kappa(\bar g_h^n-G_-)I,
\]
we have
\begin{align}
\mathrm{LHS}
={}&
\left\langle
\psi(\tau\bar L_h^n)\delta_\tau e_u^{n+1},
\delta_\tau e_u^{n+1}
\right\rangle_h
+
\left\langle
\bar L_h^n e_u^{n+1},
\delta_\tau e_u^{n+1}
\right\rangle_h
\nonumber\\
={}&
\left\langle
\left[
\psi(\tau\bar L_h^n)
+\frac{\tau}{2}
\left(
-\varepsilon^2\Delta_h+\kappa G_-I
\right)
\right]
\delta_\tau e_u^{n+1},
\delta_\tau e_u^{n+1}
\right\rangle_h
\nonumber\\
&+
\frac{\varepsilon^2}{2\tau}
\left(
\|\nabla_h e_u^{n+1}\|_h^2
-
\|\nabla_h e_u^n\|_h^2
\right)
\nonumber\\
&+
\frac{\kappa G_-}{2\tau}
\left(
\|e_u^{n+1}\|_h^2
-
\|e_u^n\|_h^2
\right)
\nonumber\\
&+
\kappa(\bar g_h^n-G_-)
\left\langle
e_u^{n+1},
\delta_\tau e_u^{n+1}
\right\rangle_h .
\label{eq_ac_lhs_split}
\end{align}
For any eigenvalue $\lambda\ge0$ of $-\Delta_h$, setting
\[
z=\tau\bigl(\kappa\bar g_h^n+\varepsilon^2\lambda\bigr),
\]
we obtain
\[
\psi(z)
+\frac{\tau}{2}
\bigl(\kappa G_-+\varepsilon^2\lambda\bigr)
=
\psi(z)+\frac{z}{2}
-\frac{\tau\kappa}{2}(\bar g_h^n-G_-).
\]
Since
\[
\psi(z)+\frac{z}{2}\ge1,
\qquad
G_-\le \bar g_h^n\le G^+,
\]
where the first inequality follows from Lemma~\ref{lem_phi_operator_bounds} and
the factor bounds follow from \eqref{eq_ac_g_two_sided}, there exists
$\tau_0>0$, independent of $h$ and $\tau$, such that, for
$0<\tau\le\tau_0$,
\[
\psi(z)
+\frac{\tau}{2}
\bigl(\kappa G_-+\varepsilon^2\lambda\bigr)
\ge c_0>0.
\]
Hence, by spectral calculus,
\begin{equation}
\left\langle
\left[
\psi(\tau\bar L_h^n)
+\frac{\tau}{2}
\left(
-\varepsilon^2\Delta_h+\kappa G_-I
\right)
\right]
\delta_\tau e_u^{n+1},
\delta_\tau e_u^{n+1}
\right\rangle_h
\ge
c_0
\|\delta_\tau e_u^{n+1}\|_h^2.
\label{eq_ac_lhs_coercivity}
\end{equation}
Moreover, by the uniform bound on $\bar g_h^n$ and Young's inequality,
\begin{align}
\kappa(\bar g_h^n-G_-)
\left\langle
e_u^{n+1},
\delta_\tau e_u^{n+1}
\right\rangle_h
&\ge
-\frac{c_0}{2}
\|\delta_\tau e_u^{n+1}\|_h^2
-
C\|e_u^{n+1}\|_h^2.
\label{eq_ac_lhs_cross_bound}
\end{align}
Combining \eqref{eq_ac_lhs_split}--\eqref{eq_ac_lhs_cross_bound} gives
\begin{align}
\mathrm{LHS}
\ge{}&
\frac{c_0}{2}
\|\delta_\tau e_u^{n+1}\|_h^2
+
\frac{\varepsilon^2}{2\tau}
\left(
\|\nabla_h e_u^{n+1}\|_h^2
-
\|\nabla_h e_u^n\|_h^2
\right)
\nonumber\\
&+
\frac{\kappa G_-}{2\tau}
\left(
\|e_u^{n+1}\|_h^2
-
\|e_u^n\|_h^2
\right)
-
C\|e_u^{n+1}\|_h^2 .
\label{eq_ac_lhs_final_bound}
\end{align}
For the right-hand side of \eqref{eq_ac_global_A_identity},
by Lemmas~\ref{lem_phi_operator_bounds}, \ref{lem:ac_lipschitz}, and
\ref{lem:ac_predictor_error}, together with \eqref{eq_ac_residual_bound},
we have
\begin{align}
\mathrm{RHS}
&\le
\Bigl(
\|\bar N_h^n-N_{\rm ex}(t_n)\|_h
+
\|\bar N_h^{*}-N_{\rm ex}(t_{n+1})\|_h
+
\|\widetilde R_u^n\|_h
\Bigr)
\|\delta_\tau e_u^{n+1}\|_h
\nonumber\\
&\le
\frac{c_0}{4}
\|\delta_\tau e_u^{n+1}\|_h^2
+
C\Bigl(
\|\bar N_h^n-N_{\rm ex}(t_n)\|_h^2
+
\|\bar N_h^{*}-N_{\rm ex}(t_{n+1})\|_h^2
+
\|\widetilde R_u^n\|_h^2
\Bigr)\no\\
&\le
\frac{c_0}{4}
\|\delta_\tau e_u^{n+1}\|_h^2
+
C\left(
\|e_u^n\|_{H_h^1}^2
+
|e_s^n|^2
+
\tau^4+h^4
\right).
\label{eq_ac_rhs_bound}
\end{align}
Combining the above estimates for $\mathrm{LHS}$ and $\mathrm{RHS}$ gives
\begin{align}
&
\frac{\varepsilon^2}{2\tau}
\left(
\|\nabla_h e_u^{n+1}\|_h^2
-
\|\nabla_h e_u^n\|_h^2
\right)
+
\frac{\kappa G_-}{2\tau}
\left(
\|e_u^{n+1}\|_h^2
-
\|e_u^n\|_h^2
\right)
+
\frac{c_0}{4}
\|\delta_\tau e_u^{n+1}\|_h^2
\nonumber\\
&\qquad\le
C\|e_u^{n+1}\|_h^2
+
C\left(
\|e_u^n\|_{H_h^1}^2
+
|e_s^n|^2
+
\tau^4+h^4
\right).
\label{eq_ac_u_energy_error}
\end{align}

Next, we estimate the error of the auxiliary variable.
Multiplying \eqref{eq_ac_error_s} by $2e_s^{n+1}$ yields
\begin{align}
\frac{1}{\tau}
\left(
|e_s^{n+1}|^2
-
|e_s^n|^2
+
|e_s^{n+1}-e_s^n|^2
\right)
=
\mathcal J_1+\mathcal J_2+\mathcal J_3 ,
\label{eq_ac_es_identity}
\end{align}
where
\begin{align*}
\mathcal J_1
&=
-
2e_s^{n+1}
\left\langle
\frac12(g_h^n f(u_h^n)+g_h^{*} f(u_h^{*})),
\delta_\tau e_u^{n+1}
\right\rangle_h,
\\
\mathcal J_2
&=
-
2e_s^{n+1}
\left\langle
\frac12
\left[
g_h^n f(u_h^n)-g(t_n)f(u(t_n))
+g_h^{*} f(u_h^{*})-g(t_{n+1})f(u(t_{n+1}))
\right],
\delta_\tau u(t_{n+1})
\right\rangle_h,
\\
\mathcal J_3
&=
-2e_s^{n+1}\left(R_s^n+\frac{\mathcal D_s^{n+1}}\tau\right) .
\end{align*}
By Theorem~\ref{thm:mbp-gsav-etd2}, the predictor bound
\eqref{eq_predictor-etd2-mbp-bound}, the factor bounds
\eqref{eq_ac_g_two_sided}, Young's inequality, and the Cauchy--Schwarz
inequality, we have
\begin{equation}\label{eq_ac_J1_bound}
\mathcal J_1
\le
C|e_s^{n+1}|^2
+
\frac{c_0}{2}
\|\delta_\tau e_u^{n+1}\|_h^2 .
\end{equation}
By Lemma~\ref{lem:ac_lipschitz}, the prediction-stage estimate
\eqref{eq_ac_star_error_bound}, and the exact-solution regularity
\eqref{eq:exact_solution_regularity},
\begin{equation}\label{eq_ac_J2_bound}
\mathcal J_2
\le
C|e_s^{n+1}|^2
+
C
\left(
\|e_u^n\|_{H_h^1}^2
+
|e_s^n|^2
+
(\tau^2+h^2)^2
\right).
\end{equation}
By \eqref{eq_ac_residual_bound} and
Lemma~\ref{lem:ac_correction_defect},
\begin{equation}\label{eq_ac_J3_bound}
\mathcal J_3
\le
C|e_s^{n+1}|^2
+
C
\left(
\|e_u^n\|_{H_h^1}^2
+
|e_s^n|^2
+
(\tau^2+h^2)^2
\right).
\end{equation}
Substituting \eqref{eq_ac_J1_bound}--\eqref{eq_ac_J3_bound} into
\eqref{eq_ac_es_identity}, we obtain
\begin{align}
\frac{1}{\tau}(|e_s^{n+1}|^2-|e_s^n|^2)
\le
C |e_s^{n+1}|^2+
\frac{c_0}{2}
\|\delta_\tau e_u^{n+1}\|_h^2
+
C
\left(
\|e_u^n\|_{H_h^1}^2
+
|e_s^n|^2
+
(\tau^2+h^2)^2
\right).
\label{eq_ac_s_error_recursive1}
\end{align}

Multiplying \eqref{eq_ac_u_energy_error} by $2$ and combining it with
\eqref{eq_ac_s_error_recursive1}, for $0<\tau,h\le1$, we obtain
\begin{align}
&
\frac{1}{\tau}
\Bigl(
|e_s^{n+1}|^2-|e_s^n|^2
\Bigr)
+
\frac{\varepsilon^2}{\tau}
\left(
\|\nabla_h e_u^{n+1}\|_h^2
-
\|\nabla_h e_u^n\|_h^2
\right)
\nonumber\\
&\quad+
\frac{\kappa G_-}{\tau}
\left(
\|e_u^{n+1}\|_h^2
-
\|e_u^n\|_h^2
\right)
\nonumber\\
&\le
C\|e_u^{n+1}\|_h^2
+
C |e_s^{n+1}|^2
+
C\left(
\|e_u^n\|_{H_h^1}^2
+
|e_s^n|^2
+
(\tau^4+h^4)
\right).
\label{eq_ac_combined_error}
\end{align}
Multiplying \eqref{eq_ac_combined_error} by $\tau$ and using
$\tau\le1$, we further obtain
\begin{align}
&
|e_s^{n+1}|^2
+
\varepsilon^2\|\nabla_h e_u^{n+1}\|_h^2
+
\kappa G_-\|e_u^{n+1}\|_h^2
\nonumber\\
&\le
|e_s^n|^2
+
\varepsilon^2\|\nabla_h e_u^n\|_h^2
+
\kappa G_-\|e_u^n\|_h^2
\nonumber\\
&\quad+
C\tau
\left(
\|e_u^{n+1}\|_h^2
+
|e_s^{n+1}|^2
\right)
\nonumber\\
&\quad+
C\tau
\left(
\|e_u^n\|_{H_h^1}^2
+
|e_s^n|^2
+
(\tau^4+h^4)
\right).
\label{eq_ac_combined_error_step}
\end{align}
Define the error energy by
\[
\mathscr E_e^n
:=
|e_s^n|^2
+
\varepsilon^2\|\nabla_h e_u^n\|_h^2
+
\kappa G_-\|e_u^n\|_h^2.
\]
Since $\varepsilon>0$, $\kappa>0$, and $G_->0$, we have
\[
\|e_u^n\|_{H_h^1}^2+|e_s^n|^2
\le C\mathscr E_e^n.
\]
Hence, \eqref{eq_ac_combined_error_step} implies
\begin{align}
\mathscr E_e^{n+1}
&\le
\mathscr E_e^n
+
C\tau \mathscr E_e^{n+1}
+
C\tau\mathscr E_e^n
+
C\tau(\tau^4+h^4),
\end{align}
and therefore
\begin{equation}
(1-C\tau)\mathscr E_e^{n+1}
\le
(1+C\tau)\mathscr E_e^n
+
C\tau(\tau^4+h^4).
\label{eq_ac_error_gronwall_pre}
\end{equation}
Choosing $\tau_0>0$ sufficiently small such that
$C\tau_0\le1/2$, for $0<\tau\le\tau_0$ we obtain, after
renaming the generic constant $C$,
\begin{equation}
\mathscr E_e^{n+1}
\le
(1+C\tau)\mathscr E_e^n
+
C\tau(\tau^4+h^4).
\label{eq_ac_error_gronwall}
\end{equation}
Iterating \eqref{eq_ac_error_gronwall} from $n=0$ to $n=m$ gives
\begin{align}
\mathscr E_e^{m+1}
&\le
(1+C\tau)^{m+1}\mathscr E_e^0
+
C\tau(\tau^4+h^4)
\sum_{j=0}^{m}(1+C\tau)^{m-j}
\nonumber\\
&\le
e^{CT}\mathscr E_e^0
+
C(\tau^4+h^4)
\tau\sum_{j=0}^{m}(1+C\tau)^{m-j}
\nonumber\\
&\le
C\left(
\tau^4+h^4
\right),
\qquad 0\le m\le N_T-1,
\label{eq_ac_error_energy_final}
\end{align}
where $\mathscr E_e^0\le Ch^4$ follows from
\eqref{eq_ac_initial_data} and
Lemma~\ref{lem:ac_periodic_grid_consistency}, and where the discrete
Gr\"onwall inequality has been used.

Taking square roots in \eqref{eq_ac_error_energy_final} gives
\begin{equation}
\|e_u^n\|_{H_h^1}
+
|e_s^n|
\le
C(\tau^2+h^2),
\qquad 0\le n\le N_T.
\label{eq_ac_final_error_bound}
\end{equation}
\end{proof}

\section{Numerical experiments}\label{section6}

In this section, we examine the accuracy and the structure-preserving
properties of the proposed GSAV--ETD2 method.  Unless otherwise stated, the
computational domain is $\Omega=(0,2\pi)^2$ with periodic boundary conditions,
and the standard double-well potential is used:
\begin{equation}\label{eq_numerical_ac_model}
 F(u)=\frac14(u^2-1)^2,
 \qquad f(u)=u-u^3,
 \qquad \beta=1.
\end{equation}
In this section, $N$ denotes the number of grid points in each coordinate
direction, so that $h=2\pi/N$; this numerical-resolution notation is distinct
from the symbol $M$ used in Section~\ref{section2}.
The spatial discretization is the second-order central-difference method
introduced in Section~\ref{section2}.  We take $\sigma(r)=e^r$ and
$\kappa=2$, which is consistent with the sufficient MBP condition
$\kappa\ge\|f'\|_{L^\infty([-1,1])}$.  The auxiliary variable is initialized by
$s_h^0=E_{1h}(u_h^0)$.  For the diagnostics below, we distinguish the
discrete physical energy and the modified GSAV energy by
\begin{equation}\label{eq_numerical_energy_diagnostics}
 E_{h,\mathrm{phys}}^n
 :=\frac{\varepsilon^2}{2}\|\nabla_h u_h^n\|_h^2
   +E_{1h}(u_h^n),
 \qquad
 E_{h,\mathrm{mod}}^n
 :=\mathcal E_h(u_h^n,s_h^n).
\end{equation}
Theorem~\ref{thm:modified_energy_stability} guarantees the decay of
$E_{h,\mathrm{mod}}^n$; the physical energy is also recorded to assess how
closely the auxiliary variable tracks the bulk energy.
The label ``original energy'' retained in some plot legends refers precisely
to $E_{h,\mathrm{phys}}^n$ in \eqref{eq_numerical_energy_diagnostics}.

\subsection{Temporal and spatial convergence}

We first verify the convergence rate for smooth periodic data over the
fixed interval $[0,1]$.  Since no closed-form solution is available, the
errors are measured by self-convergence. More precisely, in the temporal test
we use
\[
e_{u,\tau}:=u_{h,\tau}(1)-u_{h,\tau/2}(1),
\qquad
e_{s,\tau}:=s_{h,\tau}(1)-s_{h,\tau/2}(1),
\]
where the two phase fields are evaluated on the same fine grid. In the
spatial test, letting $\mathcal R_h$ denote restriction from the grid of mesh
size $h/2$ to that of mesh size $h$, we set
\[
e_{u,h}:=u_{h,\tau}(1)-\mathcal R_hu_{h/2,\tau}(1),
\qquad
e_{s,h}:=s_{h,\tau}(1)-s_{h/2,\tau}(1).
\]
The latter test uses
$N=16,32,64,128$ grid points in each
coordinate direction, so that $h=2\pi/N$, together with a sufficiently small
time step.  The interface parameter is $\varepsilon=0.08$ in both tests.

\begin{figure}[htbp]
  \centering
  \includegraphics[width=0.49\textwidth]{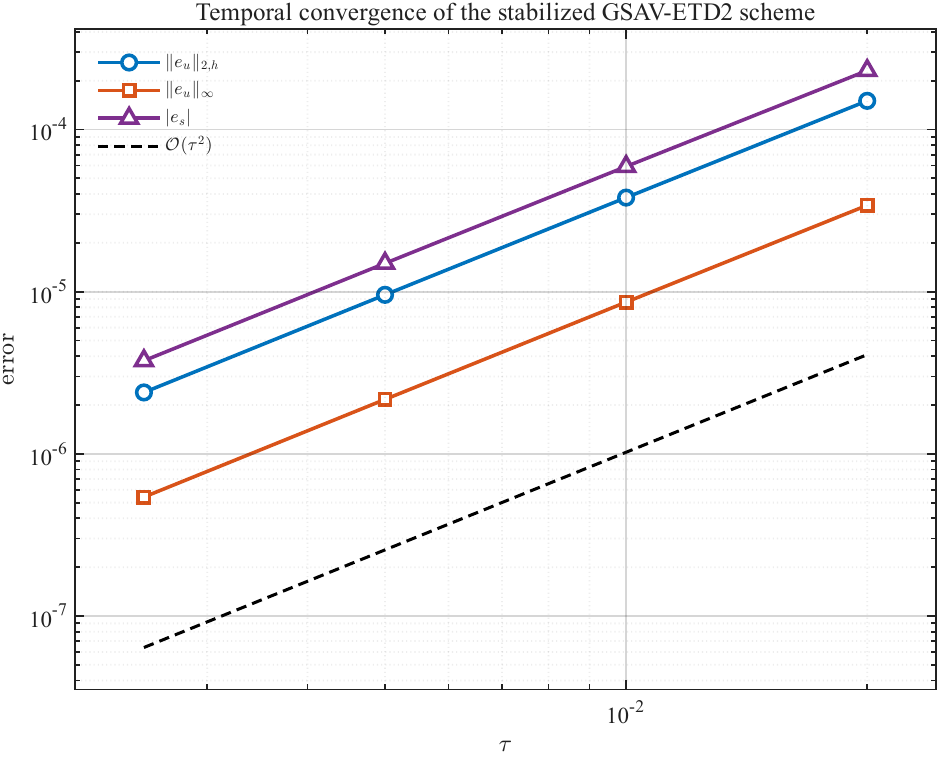}\hfill
  \includegraphics[width=0.49\textwidth]{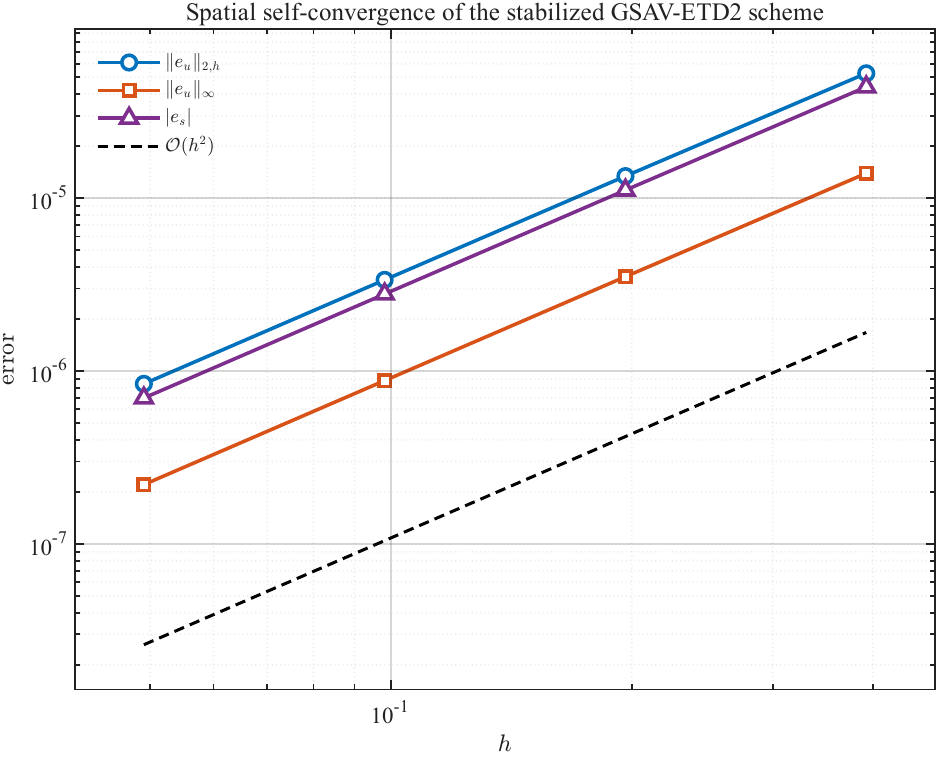}
  \caption{Self-convergence of the stabilized GSAV--ETD2 method with
  $\varepsilon=0.08$.  Left: temporal errors in
  $\|e_{u,\tau}\|_h$, $\|e_{u,\tau}\|_\infty$, and $|e_{s,\tau}|$.
  Right: the corresponding errors based on $e_{u,h}$ and $e_{s,h}$. The dashed lines indicate the
  reference slopes $O(\tau^2)$ and $O(h^2)$, respectively.}
  \label{fig:convergence_orders}
\end{figure}

As shown in Figure~\ref{fig:convergence_orders}, all three temporal error
curves are asymptotically parallel to the $O(\tau^2)$ reference line.  The
spatial errors exhibit the same behavior with respect to $h^2$.  These results
are consistent with the $H_h^1$ and scalar error estimate in
Theorem~\ref{thm_global_error} and with the second-order consistency of the
central-difference discretization. The observed second-order maximum-norm
rate is an additional empirical result and is not asserted by
Theorem~\ref{thm_global_error}. The agreement in both the phase variable and
the auxiliary variable also indicates that the scalar correction does not
reduce the order of the primary approximation.
\begin{figure}[htbp]
  \centering
  \includegraphics[width=\textwidth]{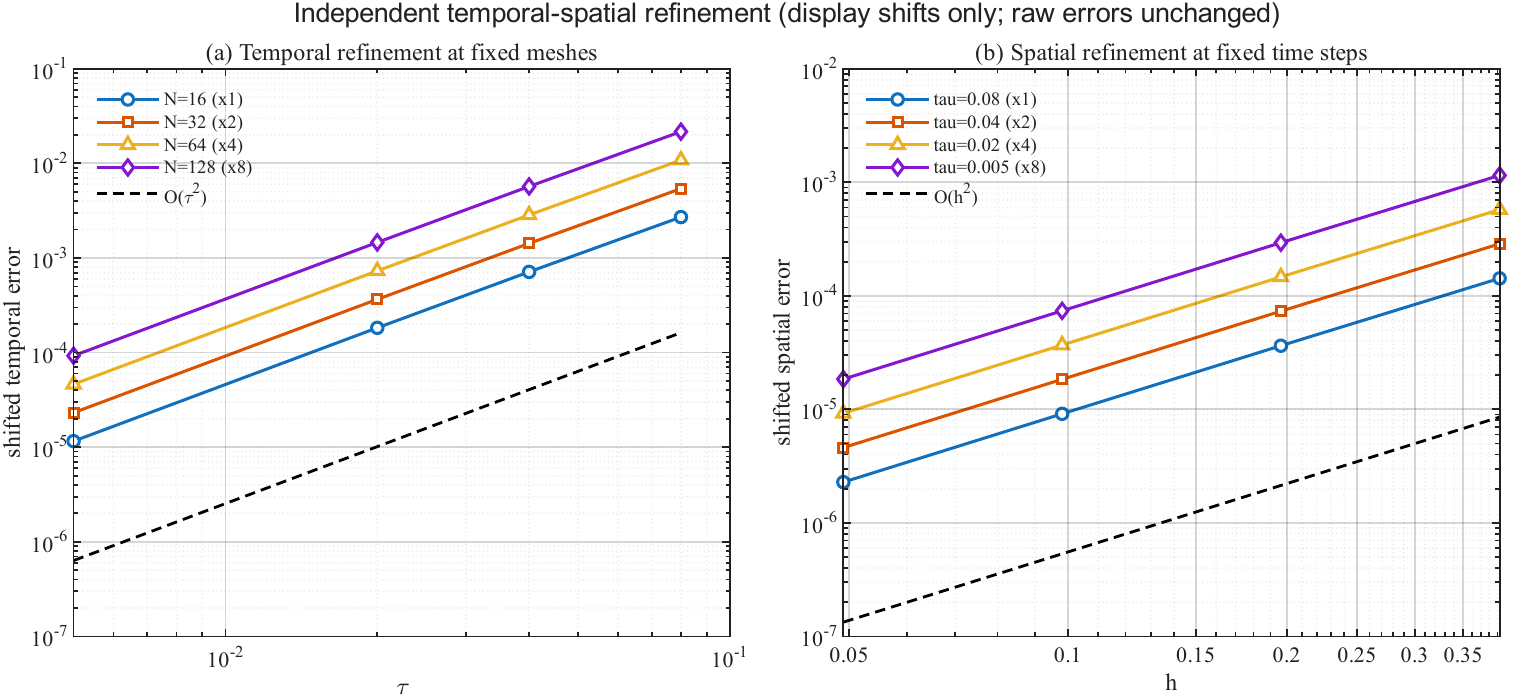}
  \caption{Independent temporal and spatial refinement tests.
(a) Temporal errors for fixed spatial meshes $N=16,32,64,128$;
(b) spatial errors for fixed time steps
$\tau=0.08,0.04,0.02,0.005$. For visibility, the four curves in each panel
are multiplied by $1$, $2$, $4$, and $8$, respectively; the slopes are
computed from the unshifted errors.}
  \label{fig:convergence_orders2}
\end{figure}

Unlike the self-convergence test in
Figure~\ref{fig:convergence_orders}, Figure~\ref{fig:convergence_orders2}
examines each discretization parameter over several independently fixed
values of the other parameter. For every fixed spatial mesh in panel (a), the
temporal error decreases with a slope consistent with $O(\tau^2)$. For every
fixed time step in panel (b), the spatial error follows the reference slope
$O(h^2)$. Thus, the second experiment confirms that the observed orders are
maintained under independent temporal and spatial refinements and are not an
artifact of a single fixed resolution.

\subsection{Large-step MBP and energy comparison}

\begin{figure}[htbp]
  \centering
  \includegraphics[width=\textwidth]{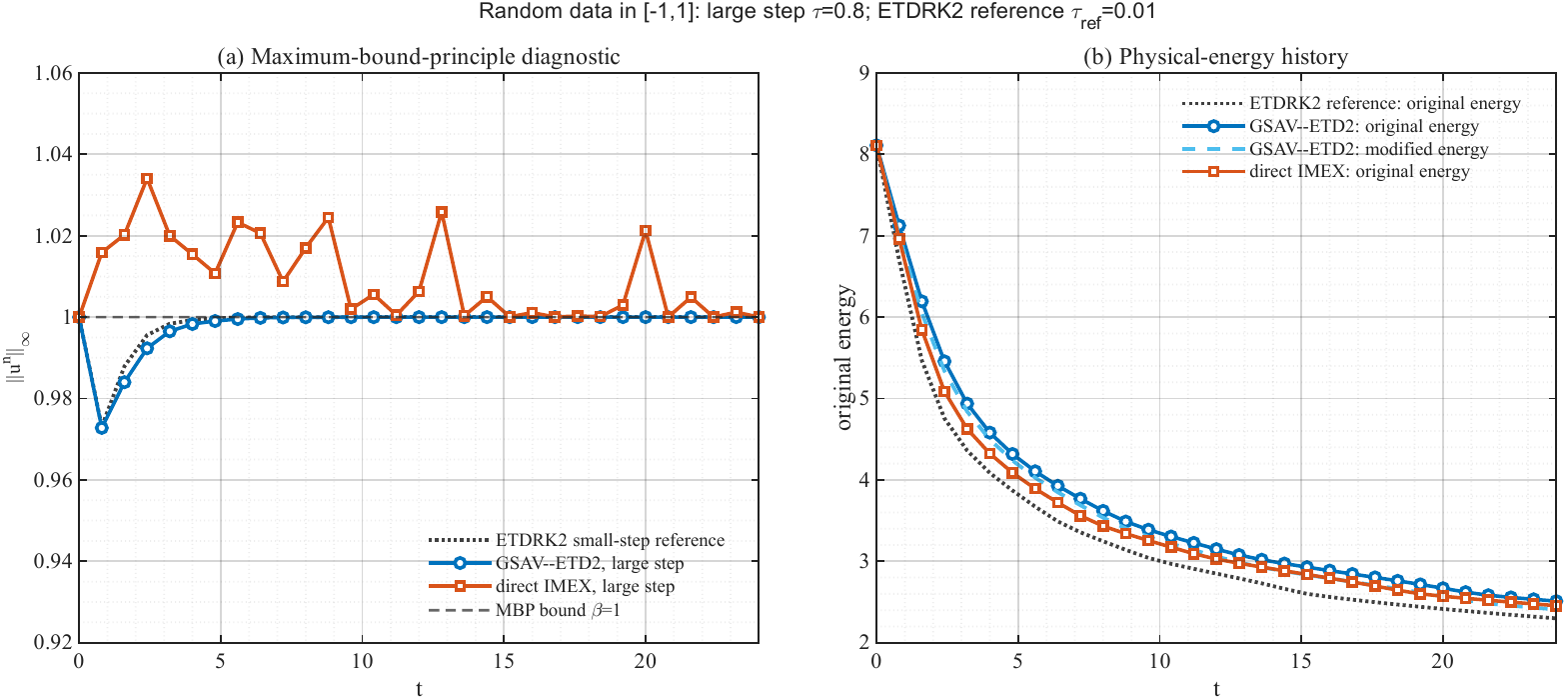}
  \caption{MBP and energy-dissipation tests for random initial data
$u_h^0\in[-1,1]$ with $\varepsilon=0.04$ on $0\le t\le20$. The large-step
schemes use $\tau=0.8$, while the ETDRK2 reference uses
$\tau_{\mathrm{ref}}=0.01$.
(a) Evolution of $\|u_h^n\|_\infty$.
(b) Evolution of the discrete energies.}
   \label{fig:mbp_energy_comparison}
\end{figure}
Figure~\ref{fig:mbp_energy_comparison} illustrates the behavior of the
proposed GSAV--ETD2 scheme for random initial data under a relatively large
time step. At each grid point, the initial value is sampled independently
from $\operatorname{Unif}[-1,1]$, and the same fixed realization is used for
all three computations.
Here the direct IMEX comparison is
\[
(I-\tau\varepsilon^2\Delta_h)u_h^{n+1}
=u_h^n+\tau f(u_h^n),
\]
whereas the ETDRK2 reference is the standard two-stage update
\begin{align*}
u_{h,\mathrm{ref}}^{*}
={}&e^{-\tau(\kappa I-\varepsilon^2\Delta_h)}u_{h,\mathrm{ref}}^n
+\tau\varphi_1\!\left(\tau(\kappa I-\varepsilon^2\Delta_h)\right)
\bigl(f(u_{h,\mathrm{ref}}^n)+\kappa u_{h,\mathrm{ref}}^n\bigr),\\
u_{h,\mathrm{ref}}^{n+1}
={}&u_{h,\mathrm{ref}}^{*}
+\tau\varphi_2\!\left(\tau(\kappa I-\varepsilon^2\Delta_h)\right)
\bigl[f(u_{h,\mathrm{ref}}^{*})+\kappa u_{h,\mathrm{ref}}^{*}
-f(u_{h,\mathrm{ref}}^n)-\kappa u_{h,\mathrm{ref}}^n\bigr].
\end{align*}
As shown in Figure~\ref{fig:mbp_energy_comparison}(a),
the numerical solution produced by the GSAV--ETD2 scheme remains within the
prescribed bound $\beta=1$ throughout the simulation, in agreement with the
discrete MBP. In contrast, the direct IMEX discretization
exhibits noticeable overshoots above the bound when the same large time step
is used.

Figure~\ref{fig:mbp_energy_comparison}(b) displays the corresponding energy
histories. Both the physical and modified energies of the GSAV--ETD2 scheme
decrease throughout the computation and remain close to the small-step
ETDRK2 reference. These results indicate that the proposed method retains
both the MBP and the dissipative energy behavior even
for relatively large time steps.
\subsection{Spinodal decomposition: MBP and energy dissipation}
We next consider a long-time spinodal decomposition starting from a small
random perturbation of the unstable state.  At every grid point, the initial
value is sampled independently from the uniform distribution
$\operatorname{Unif}[-0.05,0.05]$.  We use $\varepsilon=0.08$, a
$256\times256$ grid, and integrate to $T=400$.  The same realization of the
initial data is used in the structural diagnostics and in the phase plots.

\begin{figure}[htbp]
  \centering
  \includegraphics[width=\textwidth]{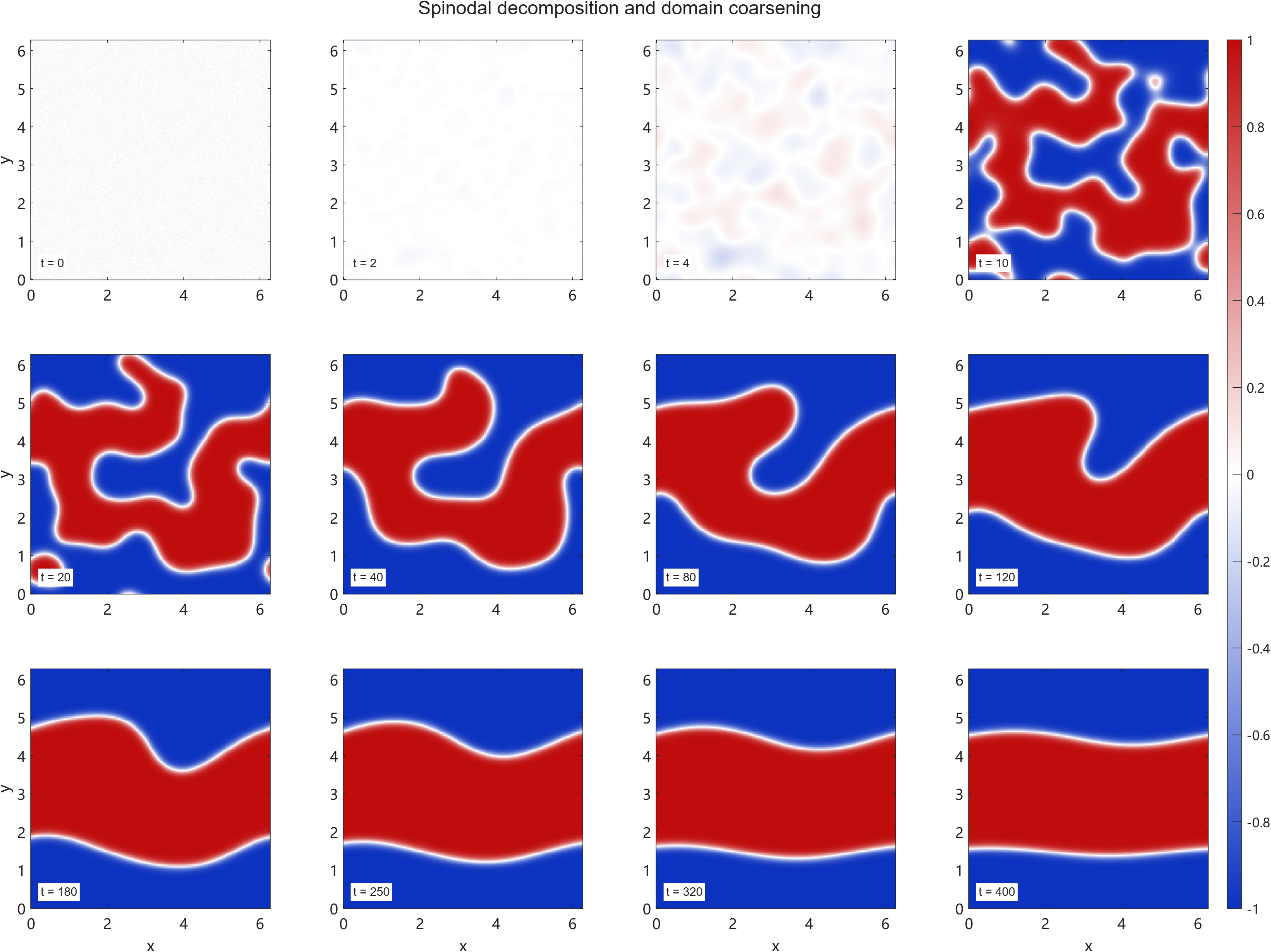}
  \caption{Snapshots of spinodal decomposition and domain coarsening at
  $t=0,2,4,10,20,40,80,120,180,250,320$, and $400$.  The color scale is fixed
  to $[-1,1]$ in all panels.}
  \label{fig:spinodal_snapshots}
\end{figure}

The snapshots in Figure~\ref{fig:spinodal_snapshots} resolve the expected
sequence of events.  The small initial fluctuation first separates rapidly
into the two stable phases $u\approx\pm1$.  Thin interconnected domains then
merge into progressively larger structures, while short interfaces disappear
and the energy continues to decrease.  At late times, the evolution is
dominated by slow domain coarsening rather than by spurious grid-scale
oscillations.

\begin{figure}[htbp]
  \centering
  \includegraphics[width=\textwidth]{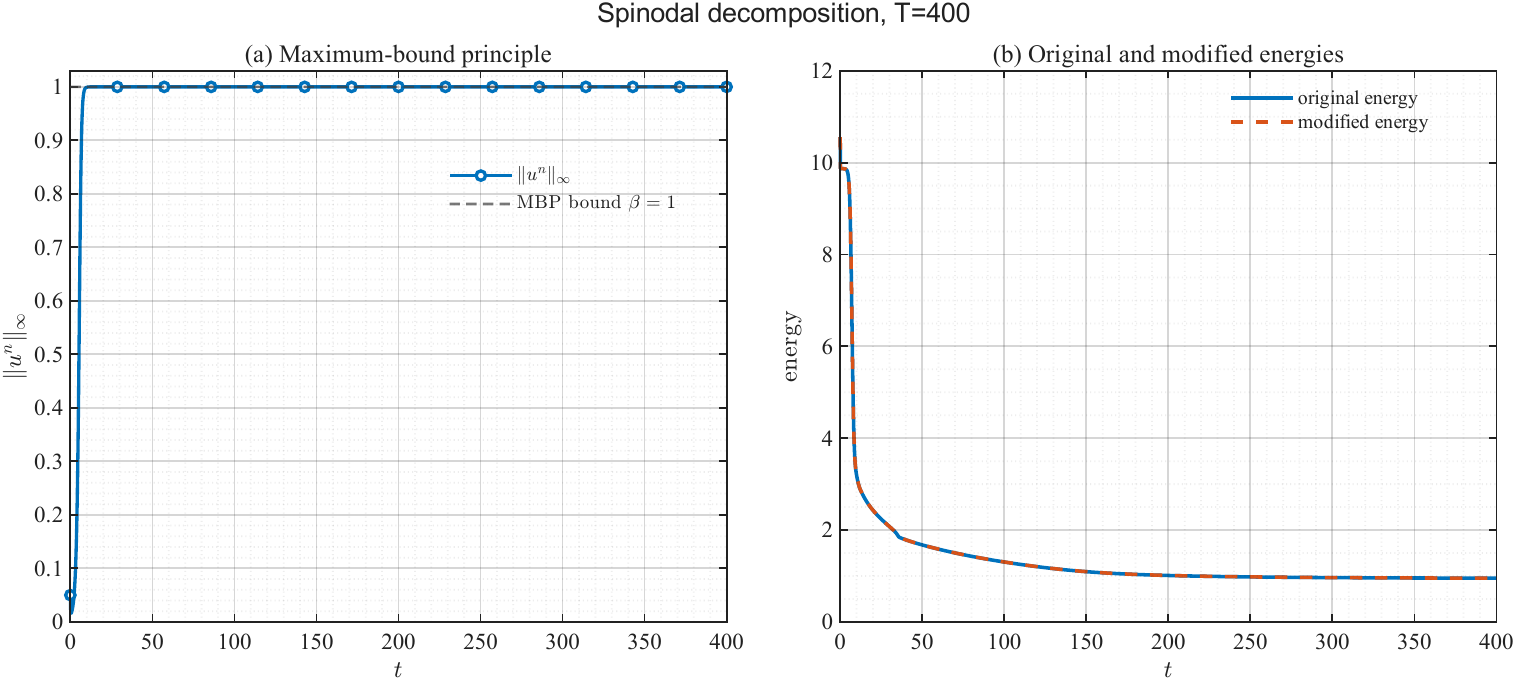}
  \caption{Spinodal decomposition with $\varepsilon=0.08$ up to $T=400$.
  Left: the discrete maximum norm and the MBP threshold $\beta=1$.
  Right: the physical energy $E_{h,\mathrm{phys}}^n$ and the modified energy
  $E_{h,\mathrm{mod}}^n$.}
  \label{fig:spinodal_structure}
\end{figure}

Figure~\ref{fig:spinodal_structure} shows that
$\|u_h^n\|_\infty\le1$ throughout the computation, in agreement with
Theorem~\ref{thm:mbp-gsav-etd2}.  Both energy curves decrease monotonically and
are visually almost indistinguishable. Thus, in addition to the rigorous
modified energy law, this experiment shows that the auxiliary and physical
energies remain close over this integration interval.

\subsection{Mean-curvature-driven interface motion}

The next test examines whether the method reproduces the sharp-interface
kinematics of the Allen--Cahn equation.  Let
\[
 r(x,y)=\sqrt{(x-\pi)^2+(y-\pi)^2}.
\]
For the normalization in \eqref{eq_numerical_ac_model}, a circular interface
of radius $R(t)$ obeys, to leading order,
\begin{equation}\label{eq_curvature_radius_law}
 R(t)^2\simeq R(0)^2-2\varepsilon^2t.
\end{equation}
Equivalently, the natural comparison variable is the scaled time
$\widehat t=\varepsilon^2t$.  The simulation in this subsection uses
$\varepsilon=0.08$ and a $256\times256$ grid.

\subsubsection{Shrinking circular droplet}

We initialize a diffuse circular droplet by
\begin{equation}\label{eq_circular_droplet_initial_data}
 u_h^0(x_i,y_j)
 =\tanh\!\left(\frac{R_0-r(x_i,y_j)}{\sqrt2\,\varepsilon}\right),
 \qquad R_0=1.5,
\end{equation}
and integrate to $T=220$.

\begin{figure}[htbp]
  \centering
  \includegraphics[width=\textwidth]{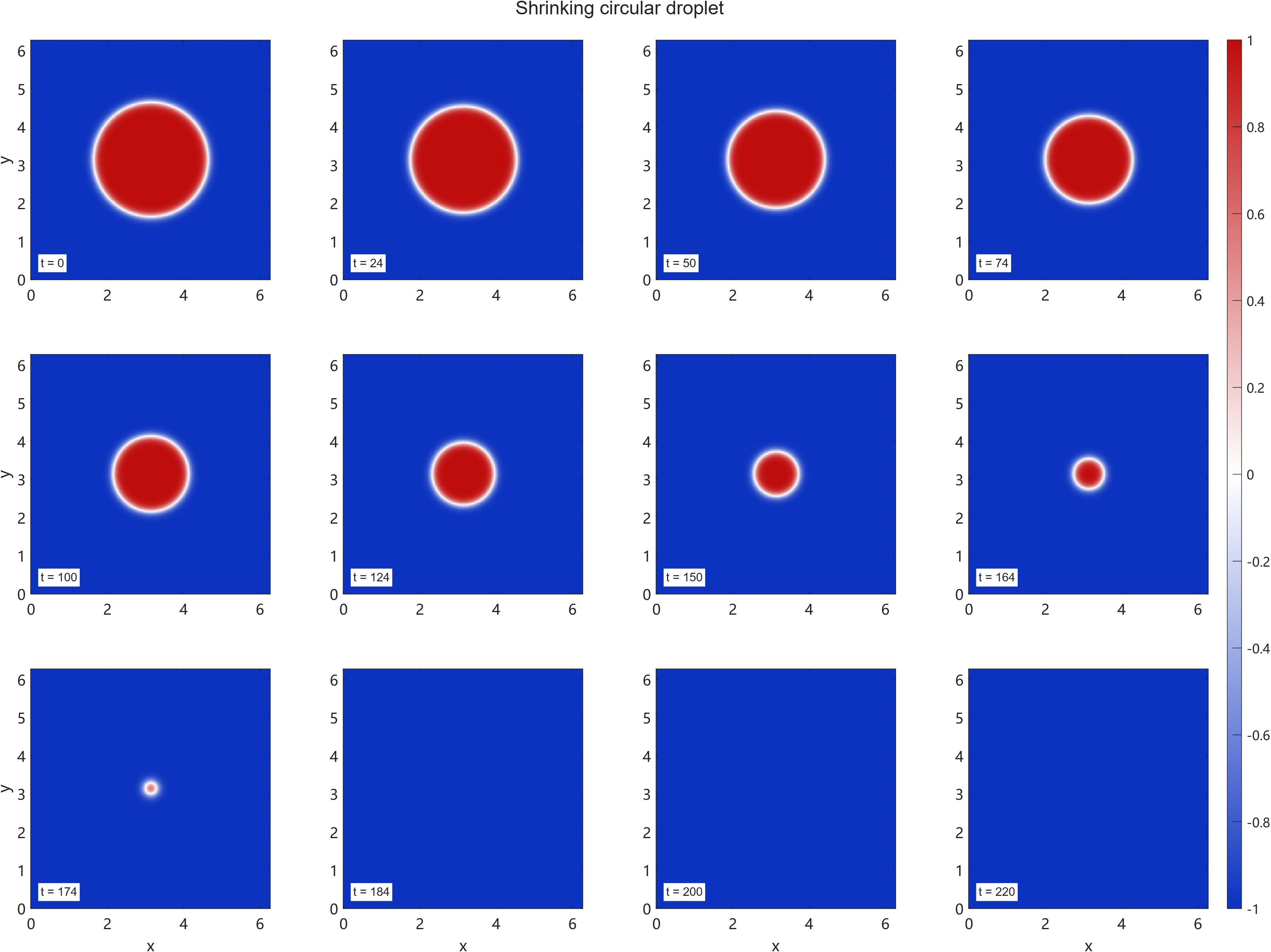}
  \caption{Phase-field snapshots of circular-droplet shrinkage at
  $t=0,24,50,74,100,124,150,164,174,184,200$, and $220$.}
  \label{fig:droplet_snapshots}
\end{figure}

\begin{figure}[htbp]
  \centering
  \includegraphics[width=\textwidth]{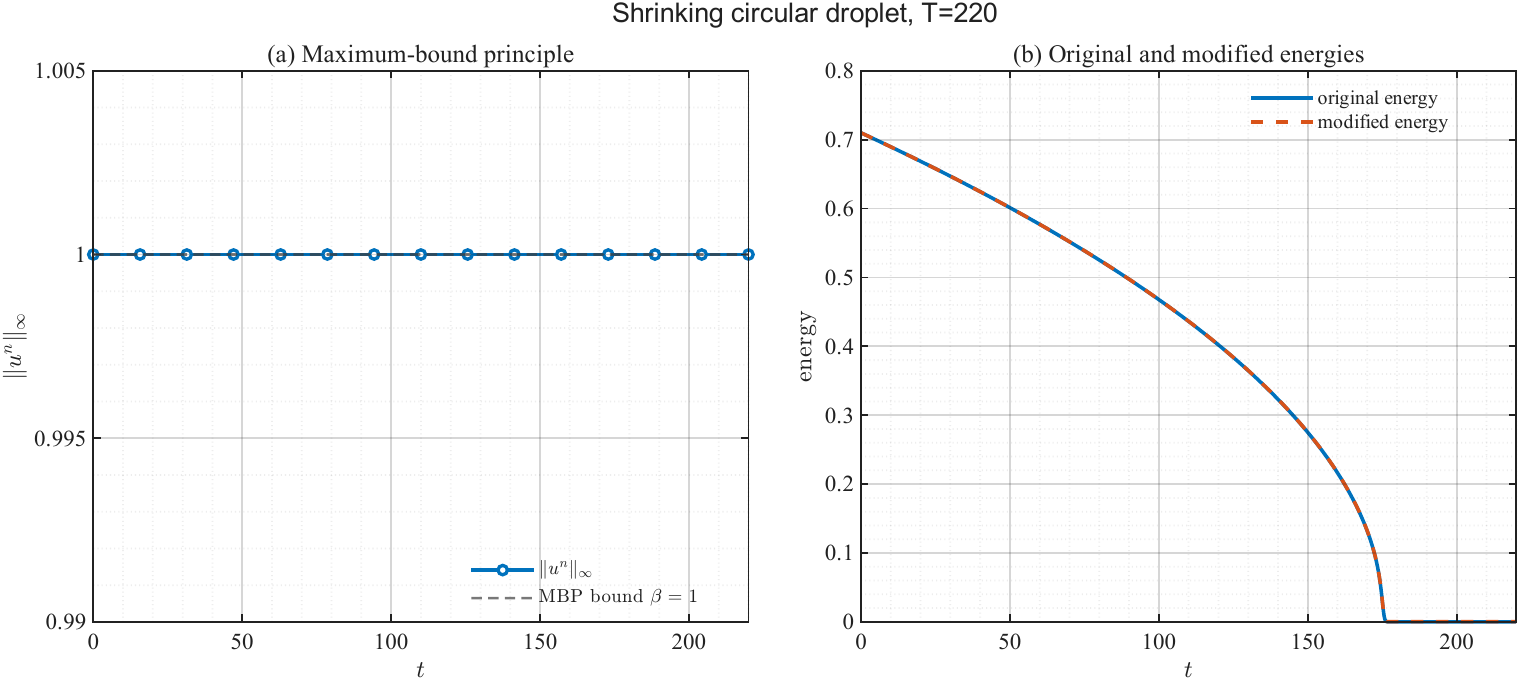}
  \caption{Shrinking circular droplet with $\varepsilon=0.08$ up to
  $T=220$.  Left: $\|u_h^n\|_\infty$ and the MBP threshold.  Right: physical
  and modified energy histories.}
  \label{fig:droplet_structure}
\end{figure}

Figures~\ref{fig:droplet_structure} and~\ref{fig:droplet_snapshots} show a
nearly circular interface throughout the evolution, followed by complete
extinction without visible overshoot.  Formula
\eqref{eq_curvature_radius_law} predicts the extinction time
\[
 T_{\rm ext}^{\rm MC}=\frac{R_0^2}{2\varepsilon^2}\approx175.78.
\]
The numerical droplet disappears between the snapshots at $t=174$ and
$t=184$, consistently with this prediction.  The two energy curves decrease
together and reach their homogeneous-state value after extinction, whereas
the maximum norm remains bounded by one.

\section{Conclusion}\label{section7}
We have developed a stabilized GSAV--ETD2 method for Allen--Cahn-type gradient
flows.  The stagewise construction combines exact treatment of the linear
dissipative operator with an auxiliary-variable correction tailored to the
second-order ETD update.  The resulting method is linear, satisfies an
unconditional modified energy dissipation law, and preserves the maximum
bound principle under an explicit condition on the stabilization parameter.
These estimates also provide the uniform bounds needed for the optimal fully
discrete error analysis.

The numerical results are consistent with second-order convergence in both
time and space, preservation of the invariant interval, and monotone decay of the
modified energy.  Long-time spinodal decomposition is captured without
spurious oscillations, while the circular-droplet test reproduces the
mean-curvature prediction for droplet extinction. A fully
controlled comparison over decreasing values of $\varepsilon$, with the
interface resolution and scaled observation times held comparable, is a
natural next step toward a quantitative sharp-interface convergence study.
\vskip 0.2cm
    {\bf Acknowledgements.} G. Ji is partially supported by the National Natural Science Foundation of China (Grant No. 12471363).

\bibliographystyle{siam}
\bibliography{S0362546X14002934}

\end{document}